\documentclass[reqno]{amsart}
\usepackage{amsmath, amsthm, amscd, amssymb, amsfonts, amsbsy}
\usepackage{latexsym, color, enumerate}
\usepackage{mathrsfs}
\usepackage{pxfonts}
\usepackage{bbm}
\usepackage{mathtools}
\usepackage[dvipsnames]{xcolor}
\usepackage{tikz}
\usepackage{enumerate}
\usepackage{todonotes}
\usepackage{marginnote}

\theoremstyle{plain}
\newtheorem{theorem}[equation]{Theorem}
\newtheorem{lemma}[equation]{Lemma}

\theoremstyle{definition}

\theoremstyle{remark}

\DeclareMathOperator{\sgn}{sgn}
\DeclareMathOperator{\ran}{ran}

\numberwithin{equation}{section}

\providecommand{\ip}[1]{\langle#1\rangle}

\providecommand{\abs}[1]{\lvert#1\rvert}
\providecommand{\Abs}[1]{\left\lvert#1\right\rvert}

\providecommand{\norm}[1]{\lVert#1\rVert}

\renewcommand{\vec}[1]{\boldsymbol{#1}}

\newcommand{\opH}{\mathsf{H}}
\newcommand{\opI}{\mathsf{I}}

\newcommand{\opM}{\mathsf{M}}
\newcommand{\opP}{\mathsf{P}}
\newcommand{\opR}{\mathsf{R}}
\newcommand{\opS}{\mathsf{S}}

\begin{document}
\title[Sign-changing Poisson kernel]
{A sign-changing Poisson kernel for a non-symmetric elliptic operator in a bounded domain}

\author[S. Kim]{Seick Kim}

\address[S. Kim]{Department of Mathematics, Yonsei University,
50 Yonsei-Ro, Seodaemun-gu, Seoul 03722, Republic of Korea}

\address[S. Kim]{Intelligent Computational Science Institute (IN2CSI),
Yonsei University, Seoul 03722, Republic of Korea}

\email{kimseick@yonsei.ac.kr}



\begin{abstract}
We construct an explicit sign-changing Poisson kernel for a uniformly elliptic divergence form operator in the unit disk.
The coefficient matrix is non-symmetric, and its skew-symmetric part has a jump discontinuity across a fixed diameter, with the size of the jump determined by $k\in\mathbb R$.
For each $1<p<\infty$, the associated $L^p$ Dirichlet problem undergoes a sharp transition as $k$ crosses a $p$-dependent threshold.
In the unique solvability regime, the solution operator preserves positivity and is represented by a nonnegative Poisson kernel.
Beyond the threshold, the homogeneous problem admits a nontrivial solution, and for real-valued boundary data, every bounded linear selection from this nonunique solution family is represented by a sign-changing kernel.
Thus, even in bounded domains, uniform ellipticity and the maximum principle in the energy class do not prevent Poisson kernels from changing sign.
Using a Riemann--Hilbert formulation, we obtain explicit solution formulas and show that the factorization index governs both nonuniqueness and sign change.
To our knowledge, this is the first explicit example of this counterintuitive phenomenon in a bounded domain.
\end{abstract}

\maketitle

\section{Introduction}
The solvability of boundary value problems for second-order elliptic equations in divergence form,
\[
Lu=\operatorname{div}(\mathbf A\nabla u)=0,
\]
is a central topic in the theory of elliptic partial differential equations.
For real, bounded, uniformly elliptic coefficients, the foundational interior theory was established by De Giorgi \cite{DeGiorgi1957}, Nash \cite{Nash1958}, and Moser \cite{Moser1961}.
In the symmetric divergence-form setting, Littman, Stampacchia, and Weinberger \cite{LSW63} studied the associated Green's function and showed, in particular, that the regular boundary points for such operators are exactly those for the Laplacian.
These results brought the Dirichlet problem for elliptic equations with rough coefficients into close analogy with the classical theory and led naturally to the study of the associated $L$-harmonic measure; see, for instance, \cite{CFMS}.

Boundary value problems require information beyond interior regularity.
The modern $L^p$ theory was initiated by Dahlberg \cite{Dahlberg1977,Dahlberg1979} and by Jerison and Kenig \cite{JK1981} for harmonic functions in Lipschitz domains, using non-tangential maximal estimates and almost everywhere non-tangential convergence.
It was later extended to elliptic equations, particularly those with symmetric coefficients independent of the transversal variable; see, for example, \cite{KP1993,Kenig1994}.

The situation becomes substantially more delicate for non-symmetric coefficients.
Nonetheless, the De Giorgi--Moser--Nash theory remains valid in the non-symmetric setting.
Moreover, Gr\"uter and Widman \cite{GW82} extended the result of Littman, Stampacchia, and Weinberger \cite{LSW63} to non-symmetric operators.
Thus the main difficulty arises neither from a failure of interior regularity nor from a failure of the maximum principle.
It lies instead in the fact that uniform ellipticity involves only the symmetric part of $\mathbf A$, whereas the skew-symmetric part may be large and can strongly affect boundary behavior.

Small non-symmetric perturbations of operators for which the $L^2$ Dirichlet problem is solvable often preserve $L^2$ solvability; see \cite{AAAHK}.
By contrast, large non-symmetric perturbations can lead to phenomena with no counterpart in the symmetric theory.
A particularly instructive example is the matrix
\begin{equation} \label{eq:mat_A}
\mathbf A(x,y)=
\begin{pmatrix}
1 & k\,\operatorname{sgn}(x)\\
-k\,\operatorname{sgn}(x) & 1
\end{pmatrix},
\qquad k\in\mathbb R.
\end{equation}
This matrix is uniformly elliptic for every $k$, since $\xi\cdot\mathbf A\xi=\abs{\xi}^2$, but its skew-symmetric part has a jump discontinuity across the vertical line.
Kenig, Koch, Pipher, and Toro \cite{KKPT} showed that, for this model in the half-plane, $L^2$ solvability may fail even though the $L^p$ Dirichlet problem remains solvable for sufficiently large $p$.
See also Hofmann, Kenig, Mayboroda, and Pipher \cite{HKMP} for related square-function and non-tangential maximal estimates in $L^p$ in higher-dimensional half-spaces.
Axelsson \cite{Axelsson} analyzed the same coefficient structure by first-order semigroup methods in the half-plane setting.

The present paper revisits this model in the unit disk and isolates a phenomenon not captured by the usual formulation of the $L^p$ Dirichlet problem.
In the classical theory, the maximum principle and uniqueness yield a positive $L$-harmonic measure.
Under $L^p$ solvability, this measure is absolutely continuous with respect to surface measure, and its density, the Poisson kernel, is nonnegative and satisfies a reverse H\"older bound.
We show that other natural solution branches for a non-symmetric elliptic operator need not share this property.
For the same uniformly elliptic operator and the same boundary data, a positivity-preserving branch in the energy class can coexist with branches outside that class whose Poisson kernels change sign.
Thus, the sign change reflects neither a failure of ellipticity nor a failure of the maximum principle, but rather nonuniqueness beyond the energy class and dependence on the selected branch.
Even when solutions attain the prescribed boundary data and satisfy the expected non-tangential estimates, the solution operator need not be unique, and some natural branches are represented by sign-changing kernels.

Our aim is to identify the mechanism responsible for this phenomenon.
In the bounded domain setting considered here, we describe the distinct solution branches through the spectrum of the boundary operator and the index of the associated Riemann--Hilbert factorization.
This shows that sign-changing Poisson kernels can also arise in a bounded domain, although the underlying branch structure differs essentially from Axelsson's half-plane model.

Let $\mathbb D$ denote the unit disk, and consider the Dirichlet problem
\[
\operatorname{div}(\mathbf A\nabla u)=0 \quad\text{in }\mathbb D,
\qquad
u=g \quad\text{on }\partial\mathbb D,
\]
where $\mathbf A$ is given by \eqref{eq:mat_A}.
By reformulating the equation as a first-order Dirac system and applying Hardy decompositions on the unit circle, we obtain the representation
\begin{equation}	\label{eq0006thu}
u(r,\,\cdot\,)=\mathsf{P_r}g-k[\mathsf{P_r},\mathsf{M_s}]\tilde{g},
\qquad \tilde{g}=(\mathsf{I}-\mathsf{R})\psi.
\end{equation}
Here, $\mathsf{P_r}$ is the classical Poisson extension operator, $\mathsf{M_s}$ denotes multiplication by
\[
\mathsf{s}(\theta):=\sgn(\cos\theta),
\]
$[\mathsf{P_r},\mathsf{M_s}]$ is the corresponding commutator, and $\mathsf{R}$ is the reflection operator on the circle.
The function $\psi$ is determined by the boundary equation
\begin{equation} \label{eq0009thu}
(2i\mathsf{I}+k\mathsf{T})\psi=\mathsf{B}g,
\end{equation}
where $\mathsf{T}=\mathsf{B}\mathsf{M_s}(\mathsf{I}-\mathsf{R})$ and  $\mathsf{B}$ is an operator associated with the Hilbert transform $\mathsf{H}$ on the unit circle, characterized by
\[
(\mathsf{I}-\mathsf{R})\mathsf{B}
=i(\mathsf{I}-\mathsf{R})\mathsf{H}.
\]
Thus, the Dirichlet problem is reduced to the analysis of the boundary operator $\mathsf{T}$.
Equivalently, after passing to the $\opR$-symmetric subspace $L^2_+$, the problem is governed by the operator
\[
\mathsf{C}=\mathsf{M}_s\mathsf{H}\vert_{L^2_+}.
\]

A central operator-theoretic result of the paper is
\[
\sigma(\mathsf C)=\overline{\mathbb D}.
\]
This spectral picture explains the sharp transition at $\abs{k}=1$ in the $L^2$ theory.
When $\abs{k}<1$, the boundary equation has a unique $L^2$ solution, given by a Neumann series.
At $\abs{k}=1$, it is not solvable for arbitrary $L^2$ data.
When $\abs{k}>1$, the associated operator is surjective but not injective, so every $L^2$ datum admits a solution, while any two solutions may differ by a nontrivial homogeneous solution.
A canonical orthogonality condition selects a bounded linear solution operator.
This loss of injectivity is the operator-theoretic source of the branching phenomenon.

Our first main result shows that whenever \eqref{eq0009thu} is solvable in $L^2(\partial\mathbb D)$, the function defined by \eqref{eq0006thu} is a weak solution satisfying weighted energy and $L^2$ non-tangential maximal estimates and attaining the boundary data in $L^2$.
In particular, $u \to g$ non-tangentially almost everywhere on $\partial\mathbb D$, placing the solution in the classical framework of the $L^2$ boundary value problem.

The second main result concerns positivity and sign change.
For $\abs{k}<1$, the solution operator obtained from the Neumann series preserves positivity, so nonnegative boundary data yield nonnegative solutions.
The Riemann--Hilbert formulation shows that the positivity-preserving branch is not confined to the regime $\abs{k}<1$.
It remains available for larger $\abs{k}$, provided that the boundary datum has sufficient integrability.
Moreover, if the boundary datum belongs to $H^\alpha(\partial\mathbb D)$ for some $\alpha>\frac12$, then the corresponding solution belongs to $H^1(\mathbb D)$ and hence satisfies the weak maximum principle.
For each $k \neq 0$, a second branch is available under an appropriate integrability condition on the boundary datum.

The Riemann--Hilbert factorization introduced in Section~\ref{sec:RH} makes the distinction transparent.
Let
\[
X(z)=\left(\frac{1-iz}{1+iz}\right)^\mu, \qquad X(0)=1, \qquad \mu=-\frac{2\arctan k}{\pi}+N.
\]
The requirement $\mu\in(-1,1)$ restricts the possible choices of the integer $N$.
When $k\neq0$, there are two such choices, and the parity of $N$ becomes decisive.
For the $L^2$ Dirichlet problem, however, one must further require $\mu\in(-\frac12,\frac12)$, which selects the relevant branch in the $L^2$ theory.
For $\abs{k}<1$, the natural choice is $N=0$, and the usual normalization of the Schwarz integral fixes the additive constant.
For $\abs{k}>1$, the natural choice is $N=\sgn k$, and the odd parity contributes a factor of $i$ to the Riemann--Hilbert problem.
In this case, the same normalization leaves a free constant.
Upon reconstructing $u$, the free constant becomes the coefficient of a nontrivial homogeneous solution.
An additional normalization is therefore required to select one solution for each boundary datum.
The canonical orthogonality condition removes this ambiguity by eliminating the homogeneous component.

This branch structure is fundamentally different from that in Axelsson's half-plane model.
In the unit disk, the the restriction $\mu \in(-1,1)$ leaves two adjacent indices, $N=0$ and $N=\sgn k$, differing by one.
In the half-plane, the corresponding exponent $\alpha$ satisfies $\tan(\pi\alpha/2)=k$ and is therefore determined modulo $2$.
Thus, an index shift by one is unavailable, while the admissible range $\alpha\in(-2,1)$ yields two branches whose indices differ by two.

The effect of the nonzero index on the resulting solution is also explicit.
In particular, the datum $g\equiv1$ does not yield $u\equiv1$, but rather a nonconstant solution.
The corresponding kernel therefore cannot represent $L$-harmonic measure.
At the same time, the analysis explains why a positivity-preserving solution remains available for $L^p$ data when $p$ is sufficiently large, consistently with the $L^p$ solvability results in \cite{KKPT,HKMP}.

More precisely, for $1<p<\infty$, the sharp threshold is
\[
\abs{k}=\cot \big(\frac{\pi}{2p}\big).
\]
Below this threshold, the $L^p$ Dirichlet problem is uniquely solvable and the corresponding $N=0$ branch has a nonnegative Poisson kernel.
Above the threshold, the corresponding homogeneous problem has a nontrivial solution, so uniqueness fails in the same class.
Moreover, for real-valued boundary data, every bounded linear selection from the resulting nonunique family is represented by a sign-changing kernel.
When $p=2$, the threshold reduces to $\abs{k}=1$.
There is no contradiction with the maximum principle, since the standard weak maximum principle applies to real $H^1(\mathbb D)$ solutions.
For the constant datum $g \equiv 1$, however, the branch $N=\sgn k$ produces an explicit nonconstant solution that changes sign, and a direct computation shows that this solution does not belong to $H^1(\mathbb D)$.

We conclude the introduction with an overview of the paper.
In Section~\ref{sec:first_order}, we reduce the second-order equation to a first-order Dirac system and derive the solution representation \eqref{eq0006thu} together with the boundary equation \eqref{eq0009thu}.
Section~\ref{sec:weak_sol} establishes weighted energy estimates, the weak solution formulation, $L^2$ boundary convergence, non-tangential maximal estimates, and almost everywhere non-tangential convergence to the boundary data.
In Section~\ref{sec:solution}, we analyze the boundary equation \eqref{eq0009thu}, including the spectrum of $\mathsf{C}$ and the solvability dichotomy at $\abs{k}=1$.
Section~\ref{sec:small_k} treats the case $\abs{k}<1$ by elementary methods and initiates the study of the $L^p$ Dirichlet problem.
Finally, Section~\ref{sec:RH} develops the Riemann--Hilbert formulation and explains the role of the factorization index.
There, we identify the positivity-preserving branch that produces solutions in the energy class and distinguish it from the branches outside that class that give rise to sign-changing Poisson kernels.
We also establish the precise threshold at which branching occurs for $L^p$ boundary data.
The main results are summarized in Theorem~\ref{thm:kernel}.

\section{Reduction to Dirac systems and solution representation}
\label{sec:first_order}
We employ the first-order Dirac systems approach, a method introduced in \cite{AAH} in the
study of complex elliptic operators.
Following its initial development via functional calculus, this framework was extended to
elliptic systems with complex coefficients \cite{AAM2010}, as well as to weighted maximal
regularity estimates \cite{AA2011, AR12}.

In this section, we present a self-contained introduction showing how the Dirichlet problem
\begin{equation} \label{eq:dp_in_disk}
Lu=\mathrm{div}(\mathbf{A}\nabla u)=0 \quad \text{in } \mathbb{D}, \qquad
u=g \quad \text{on } \partial\mathbb{D},
\end{equation}
can be solved using this first-order approach.
Here, $\mathbb{D}$ denotes the unit disk in $\mathbb{R}^2$ and $\partial\mathbb{D}$ its boundary, the unit circle.

This section is self-contained, and no prior familiarity with first-order Dirac systems is required.

\subsection{First-order Dirac systems}

We begin by rewriting the equation
\begin{equation} \label{eq:u}
Lu = \mathrm{div}(\mathbf{A} \nabla u) = 0 \quad \text{in }\, \mathbb{D},
\end{equation}
in polar coordinates.
In the polar basis we have
\begin{equation}	\label{eq1636sun}
\mathbf{A} =
\begin{pmatrix}
1 & k \mathsf{s} \\
-k \mathsf{s} & 1
\end{pmatrix},\qquad \mathsf{s}=\mathsf{s}(\theta):=\sgn(\cos \theta),
\end{equation}
and the operator becomes
\[
\mathrm{div} (\mathbf{A} \nabla u)=
\frac{1}{r} \frac{\partial}{\partial r} \left( r\frac{\partial u}{\partial r}+k\mathsf{s} \frac{\partial u}{\partial \theta} \right) +\frac{1}{r} \frac{\partial}{\partial \theta} \left( -k\mathsf{s} \frac{\partial u}{\partial r} + \frac{1}{r} \frac{\partial u}{\partial \theta}\right).
\]
Next, we introduce the change of variables $r = e^{-t}$.
Under this transformation, equation~\eqref{eq:u} becomes
\begin{equation}		\label{eq1440mon2}
\frac{\partial}{\partial t} \left(\frac{\partial u}{\partial t}-k \mathsf{s} \frac{\partial u}{\partial \theta} \right) +\frac{\partial}{\partial \theta} \left( k\mathsf{s} \frac{\partial u}{\partial t} +\frac{\partial u}{\partial \theta}\right)=0\quad\text{in}\quad \mathbb{R}_{+}\times \partial\mathbb{D}.
\end{equation}
This can be compactly written as
\begin{equation}		\label{eq1440mon}
\mathrm{div}(\mathbf{A}^\top \nabla u) =0 \quad\text{in}\quad \mathbb{R}_{+}\times \partial\mathbb{D},
\end{equation}
where $\mathbf{A}$ is as defined in \eqref{eq1636sun}.
It is worth noting that we obtained~\eqref{eq1440mon} from~\eqref{eq:u} through the conformal change of variables
\[
(x,y) \mapsto (t,\theta),\qquad\text{with}\quad
x=e^{-t}\cos\theta, \quad y=e^{-t}\sin \theta.
\]

Following the approach in~\cite{AA2011} (see also~\cite{AAM2010, AAH}), we observe that equation~\eqref{eq1440mon} can be equivalently written as
\begin{equation}			\label{eq1549mon}
\frac{\partial}{\partial t}
\begin{pmatrix}
1 & -k \mathsf{s} \\
0 & 1
\end{pmatrix}
\binom{\frac{\partial u}{\partial t}}{\frac{\partial u}{\partial \theta}}+
\begin{pmatrix}
0 & \frac{\partial}{\partial \theta} \\
-\frac{\partial}{\partial\theta} & 0
\end{pmatrix}
\begin{pmatrix}
1 &0 \\
k \mathsf{s} & 1
\end{pmatrix}
\binom{\frac{\partial u}{\partial t}}{\frac{\partial u}{\partial \theta}}=0.
\end{equation}
Indeed, the first row of~\eqref{eq1549mon} corresponds exactly to~\eqref{eq1440mon}, while the second row reduces to the identity
\[
\frac{\partial^2 u}{\partial t \partial \theta} - \frac{\partial^2 u}{\partial \theta \partial t}=0.
\]

We now introduce the conormal gradient field
\begin{equation}			\label{eq1538mon}
\vec F:=
\begin{pmatrix}
1 & -k \mathsf{s} \\
0 & 1
\end{pmatrix}
\binom{\frac{\partial u}{\partial t}}{\frac{\partial u}{\partial \theta}}
=\left(\frac{\partial u}{\partial t} - k \mathsf{s} \frac{\partial u}{\partial \theta},\,\frac{\partial u}{\partial \theta} \right)^\top
\end{equation}
and define the matrix
\begin{equation}			\label{eq1539mon}
\mathbf{B}:=
\begin{pmatrix}
1 &0 \\
k \mathsf{s} & 1
\end{pmatrix}
\begin{pmatrix}
1 & -k \mathsf{s} \\
0 & 1
\end{pmatrix}^{-1}=
\begin{pmatrix}
1 & k \mathsf{s} \\
k \mathsf{s} & 1+k^2 \mathsf{s}^2
\end{pmatrix}.
\end{equation}
Note that $\mathbf{B}$ is independent of $t$.

With these definitions, the first-order system~\eqref{eq1549mon} takes the compact form
\begin{equation}			\label{eq:first_order}
\frac{\partial}{\partial t} \vec F+ \mathbf{D}\mathbf{B} \vec F=0 \quad\text{in}\quad \mathbb{R}_{+}\times \partial\mathbb{D},
\end{equation}
where $\mathbf{D}$ is the tangential Dirac operator
\begin{equation} \label{eq:matrix_DB}
\mathbf{D} =
\begin{pmatrix}
0 & \frac{\partial}{\partial \theta} \\
-\frac{\partial}{\partial\theta} & 0
\end{pmatrix}=
\mathbf{J}^\top \frac{\partial}{\partial \theta}
\quad \text{and}\quad \mathbf{J}=
\begin{pmatrix}
0 & -1 \\
1 & 0
\end{pmatrix}.
\end{equation}

We observe that there is a one-to-one correspondence between solutions $\vec{F}$ of~\eqref{eq:first_order} and gradients $\nabla u$ of solutions $u$ to~\eqref{eq1440mon}.
Indeed, as shown earlier, if $u$ satisfies~\eqref{eq1440mon}, then the vector field $\vec F$ defined by~\eqref{eq1538mon} satisfies~\eqref{eq:first_order}.
Conversely, suppose that $\vec F$ satisfies~\eqref{eq:first_order}.
Using~\eqref{eq1539mon} and~\eqref{eq:matrix_DB}, we may rewrite~\eqref{eq:first_order} in the form
\[
\frac{\partial}{\partial t}
\begin{pmatrix}
1 & -k \mathsf{s} \\
0 & 1
\end{pmatrix}
\begin{pmatrix}
1 & -k \mathsf{s} \\
0 & 1
\end{pmatrix}^{-1}
\vec F+
\frac{\partial}{\partial \theta}
\begin{pmatrix}
0  &1 \\
-1 & 0
\end{pmatrix}
\begin{pmatrix}
1&0 \\
k \mathsf{s}  & 1
\end{pmatrix}
\begin{pmatrix}
1 & -k \mathsf{s} \\
0 & 1
\end{pmatrix}^{-1}
\vec F=0.
\]

Define
\begin{equation}			\label{eq0606tue}
\binom{v}{w}:=
\begin{pmatrix}
1 & -k \mathsf{s} \\
0 & 1
\end{pmatrix}^{-1}
\vec F=
\begin{pmatrix}
1 & k \mathsf{s} \\
0 & 1
\end{pmatrix}
\vec F.
\end{equation}
Then $(v, w)^\top$ satisfies the system
\begin{equation}			\label{eq2116mon}
\frac{\partial}{\partial t}
\begin{pmatrix}
1 & -k \mathsf{s} \\
0 & 1
\end{pmatrix}
\binom{v}{w}+
\frac{\partial}{\partial \theta}
\begin{pmatrix}
k \mathsf{s} &1 \\
-1 & 0
\end{pmatrix}
\binom{v}{w}=0.
\end{equation}
The second row of~\eqref{eq2116mon} yields
\[
\frac{\partial w}{\partial t}-\frac{\partial v}{\partial \theta}=0.
\]
Hence, there exists a scalar function $u$ on $\mathbb{R}_{+}\times \partial\mathbb{D}$ satisfying
\begin{equation}			\label{eq0609tue}
\frac{\partial u}{\partial t}=v, \quad \frac{\partial u}{\partial \theta}=w.
\end{equation}
Substituting these relations into the first row of~\eqref{eq2116mon} yields~\eqref{eq1440mon2}.

This completes the proof of the desired one-to-one correspondence between solutions of the first-order system~\eqref{eq:first_order} and the gradients of solutions to the second-order equation~\eqref{eq1440mon}.

\subsection{Setup for the Dirichlet problem and the dual system}
The goal of this subsection is to link the transformed Dirichlet problem
\begin{equation} \label{eq:dp}
Lu = \mathrm{div}(\mathbf{A}^\top \nabla u) = 0\quad\text{in}\quad \mathbb{R}_{+} \times \partial\mathbb{D},\qquad u(0, \cdot)=g \quad\text{on}\quad \partial\mathbb{D},
\end{equation}
with the first-order system~\eqref{eq:first_order}.
Recall that if $u$ is a solution of~\eqref{eq1440mon}, then $\vec F$ defined in \eqref{eq1538mon} satisfies \eqref{eq:first_order}.
\subsubsection*{The dual first-order system}
Following the approach in~\cite{AAM2010}, we now consider the dual first-order system
\begin{equation}			\label{eq:v}
\frac{\partial}{\partial t} \vec V +\mathbf{B} \mathbf{D} \vec V=0 \quad\text{in}\quad \mathbb{R}_{+}\times \partial\mathbb{D}.
\end{equation}
Writing $\vec V=(v, w)^\top$, this system takes the form
\[
\frac{\partial}{\partial t}
\binom{v}{w}+
\begin{pmatrix}
-k \mathsf{s} &  1\\
-(1+k^2 \mathsf{s}^2) & k \mathsf{s}
\end{pmatrix}
\frac{\partial}{\partial \theta}
\binom{v}{w}=0.
\]
In particular, from the first row we obtain
\begin{equation}			\label{eq0653tue}
\frac{\partial v}{\partial t}= k\mathsf{s} \frac{\partial v}{\partial \theta} - \frac{\partial w}{\partial \theta}.
\end{equation}

Note that $\vec F= \mathbf{D} \vec V=(\frac{\partial w}{\partial \theta}, -\frac{\partial v}{\partial \theta})^\top$ satisfies \eqref{eq:first_order}.
Hence, there exists a function $u$ such that (cf.~\eqref{eq0606tue} and~\eqref{eq0609tue})
\begin{equation}			\label{eq0657tue}
\frac{\partial u}{\partial t}=   \frac{\partial w}{\partial \theta} - k \mathsf{s} \frac{\partial v}{\partial \theta} , \quad \frac{\partial u}{\partial \theta}=-\frac{\partial v}{\partial \theta},
\end{equation}
and $u$ satisfies~\eqref{eq1440mon}.
Comparing~\eqref{eq0653tue} and~\eqref{eq0657tue}, we find that
\[
v=-u+c \quad\text{for some constant }\,c.
\]
Hence, we conclude that if $\vec V$ is a solution of~\eqref{eq:v}, then $v$, the first component of $\vec V$, is a solution of~\eqref{eq1440mon}.
This key observation was first made in~\cite{AAM2010}.

\subsubsection*{Hardy decomposition and Cauchy extension}
It is straightforward to verify that the spectrum of the first-order operator $\mathbf{B}\mathbf{D}$ is the set of all integers, $\mathbb{Z}$, and that each eigenvalue $n \in \mathbb{Z}$ has multiplicity $2$.
Indeed, note that
\[
\mathbf{B}\mathbf{D}=\mathbf{M}\frac{d}{d\theta},
\quad\text{where}\quad
\mathbf{M}:=
\mathbf{B}\mathbf{J}^\top=
\begin{pmatrix}
-k \mathsf{s} &  1\\[2pt]
-(1+k^2 \mathsf{s}^2) & k \mathsf{s}
\end{pmatrix}.
\]

It is convenient to identify $\partial\mathbb{D}$ with $\mathbb{R}/(2\pi \mathbb{Z})$ and introduce the intervals
\[
I_{(+)}:=(-\tfrac{\pi}{2}, \tfrac{\pi}{2}),\qquad I_{(-)}:=(\tfrac{\pi}{2}, \tfrac{3\pi}{2}).
\]
Note that $\mathbf{M}$ is constant on each of $I_{(+)}$ and $I_{(-)}$.
Define
\[
\mathbf{M}_+: =\begin{pmatrix}-k&1\\[2pt]-(1+k^2)&k\end{pmatrix}, \qquad
\mathbf{M}_- :=\begin{pmatrix} k&1\\[2pt]-(1+k^2)&-k\end{pmatrix}
\]
and set
\begin{equation}	\label{basis_vn}
\begin{aligned}
\vec v_n^{(+)}(\theta)&:=
\begin{cases}
\displaystyle
(-i)^{n} (\mathbf{I}-i \mathbf{M}_+)\, \binom{i}{1}\, e^{-i n\theta}
+ i^{n}(\mathbf{I}+i \mathbf{M}_+)\, \binom{i}{1}\, e^{i n\theta}, & \theta\in I_{(+)},\\[10pt]
\displaystyle
(-i)^{n} (\mathbf{I}-i \mathbf{M}_-)\, \binom{i}{1}\, e^{-i n\theta}
+ i^{n}\,(\mathbf{I}+i \mathbf{M}_-)\, \binom{i}{1}\, e^{i n\theta}, & \theta\in I_{(-)},
\end{cases}\\[5pt]
\vec v_n^{(-)}(\theta)&:= \overline{\vec v_n^{(+)}(\theta)}.
\end{aligned}
\end{equation}

\smallskip
Observe that $\mathbf{M}_{\pm}^2=-\mathbf{I}$, which implies
\[
(\mathbf{I} -i \mathbf{M}_{\pm})\mathbf{M}_{\pm} = i (\mathbf{I} - i \mathbf{M}_{\pm}),\qquad
(\mathbf{I} +i \mathbf{M}_{\pm})\mathbf{M}_{\pm} = -i (\mathbf{I} + i \mathbf{M}_{\pm}).
\]
Using these identities, one readily verifies that
\[
\mathbf{B}\mathbf{D} \vec v_n^{(\pm)} = n \vec v_n^{(\pm)}\quad\text{on each } \,I_{(\pm)}.
\]
Furthermore, it is straightforward to check that 
\[
\vec v_n^{(\pm)} \in H^1(\partial\mathbb{D}; \mathbb{C}^2).
\]
In particular, note that $\vec v_n^{(\pm)}$ is continuous at $\pi/2$ and $3\pi/2$, where the sign of $\mathsf{s}$ changes.
Therefore, the two functions $\vec v_n^{(+)}$ and $\vec v_n^{(-)}$ form a basis for the eigenspace corresponding to the eigenvalue $n \in \mathbb{Z}$.
It is straightforward to check that $\{\vec v_n^{(+)}, \vec v_n^{(-)}\}_{n \in \mathbb{Z}}$ forms a complete, orthogonal basis for the Hilbert space $L^2(\partial\mathbb{D}; \mathbb{C}^2)$.

If we write $\vec f$ as a series expansion with respect to $\{\vec v_n^{(+)}, \vec v_n^{(-)}\}_{n \in \mathbb{Z}}$, that is,
\[
\vec f(\theta)= \sum_{n=-\infty}^\infty \left(a_n\vec v_{n}^{(+)}(\theta) + b_n\vec v_{n}^{(-)}(\theta) \right),
\]
then the Hardy projection $\mathbf{E}^{+}$ and the Cauchy extension operator $\mathbf{C}_t^+$ are given by
\begin{align}
			\label{eq2037sat}
\mathbf{E}^{+} \vec f (\theta)&=\sum_{n=0}^\infty \left(a_n \vec v_{n}^{(+)}(\theta) + b_n \vec v_n^{(-)}(\theta) \right),\\
			\label{eq2038sat}
\mathbf{C}_{t}^{+}\vec f (\theta)&= \sum_{n=0}^\infty e^{-tn}\left(a_n \vec v_{n}^{(+)}(\theta) + b_n \vec v_n^{(-)}(\theta) \right).
\end{align}
Then the function $\vec V(t, \theta) =\mathbf{C}^{+}_t \vec f(\theta)$ satisfies the evolution equation
\begin{equation}			\label{eq0500sat}
\frac{\partial}{\partial t} \vec V+ \mathbf{B} \mathbf{D} \vec V=0,\qquad
\lim_{t\to 0+} \vec V(t, \cdot)= \mathbf{E}^{+}\vec f.
\end{equation}

\subsection{Boundary equations and series representation}
Setting
\[
\mathbf{C}^{+}_t \vec f=(u, \tilde{u})^\top,
\]
we observe from \eqref{eq0500sat} that $u$ is a solution of~\eqref{eq1440mon}.
Moreover,
\begin{equation}			\label{eq0920fri}
\lim_{t\to 0+} u(t,\cdot)=(\mathbf{E}^{+} \vec f)_{1},
\end{equation}
where $(\mathbf{E}^{+} \vec f )_{1}$ denotes the first component of $\mathbf{E}^{+} \vec f$.

Consequently, for a given Dirichlet boundary datum $g$, if one can find $\vec f$ such that the first component of $\mathbf{E}^{+} \vec f$ equals $g$, then the Dirichlet problem~\eqref{eq:dp} is solved.

From~\eqref{basis_vn}, the first component $\phi_n^{(+)}$ of $\vec v^{(+)}_n$ is given explicitly by
\[
\phi_n^{(+)}(\theta) =
\begin{cases}
(-i)^n (-k) e^{-i n\theta} + i^n (k+2i) e^{i n\theta}, & \theta\in I_{(+)},\\[5pt]
(-i)^n k e^{-i n\theta} + i^n (-k+2i) e^{i n\theta}, & \theta\in I_{(-)}.
\end{cases}
\]
We write this as
\begin{equation}		\label{eq1327thu}
\phi_n^{(+)}(\theta) = (-i)^n (-k \mathsf{s}) e^{-i n\theta}+i^n (k \mathsf{s}+2i)e^{i n\theta}.
\end{equation}
The first component $\phi_n^{(-)}$ of $\vec v^{(-)}_n$ is then given by
\begin{equation}		\label{eq1328thu}
\phi_n^{(-)}(\theta) = i^n(-k \mathsf{s}) e^{i n\theta}+(-i)^n (k \mathsf{s}-2i)e^{-i n\theta}.
\end{equation}

For a given $g \in L^2(\partial \mathbb{D})$, in light of \eqref{eq2037sat} and \eqref{eq0920fri}, we seek coefficients $a_n$, $b_n$ such that
\begin{equation}			\label{eq1320thu}
g=\sum_{n=0}^\infty a_n \phi_n^{(+)}+ \sum_{n=0}^\infty  b_n \phi_n^{(-)},\quad b_0=0.
\end{equation}
The functions $\phi_0^{(+)}$ and $\phi_0^{(-)}$ are linearly dependent, and we retain only one and set $b_0=0$.
In Fourier form,
\[
g(\theta)= \sum_{n=0}^\infty i^{n} \left( (k \mathsf{s}+2i)a_n - k \mathsf{s} b_n \right) e^{i n\theta} + \sum_{n=0}^\infty (-i)^{n}\left( -k \mathsf{s} a_n + (k \mathsf{s}-2i) b_n \right) e^{-i n\theta}.
\]

To express this more compactly, define the following a priori unknown functions:
\[
\alpha(\theta):=\sum_{n=0}^\infty i^{n}a_n e^{in\theta},\qquad
\beta(\theta):=\sum_{n=0}^\infty i^{n}b_n e^{in\theta}.
\]
These functions belong to the classical Hardy space $H^2_{\rm Hardy}(\partial\mathbb{D})$; see
\cite{Garnett} for details on Hardy spaces.
Introduce the reflection operator $\opR: L^2(\partial\mathbb{D}) \to L^2(\partial\mathbb{D})$ by
\[
\opR f (\theta)=f(\pi-\theta),
\]
where, here and throughout, we identify $f(\theta)$ with $f(e^{i\theta})$.
Note that
\[
\opR f(\theta)=\sum_{n=-\infty}^\infty (-1)^{n} \hat f(n)\,e^{-in\theta}.
\]
It is straightforward to see that
\begin{equation}			\label{property_r}
\opR^2=\opI,\qquad \opR1=1.
\end{equation}
Note that
\[
\opR \alpha(\theta)=\sum_{n=0}^\infty (-i)^{n}a_n e^{-in\theta}, \qquad  \opR \beta (\theta)=\sum_{n=0}^\infty (-i)^{n}b_n e^{-in\theta}.
\]
In this setting, we seek functions $\alpha$ and $\beta$ satisfying
\begin{equation}		\label{eq_for_g}
g(\theta)= 2i(\alpha-\opR\beta)+k \mathsf{s} [(\alpha-\beta)- \opR(\alpha-\beta)].
\end{equation}

Recall that we employed the change of variables $r=e^{-t}$ to transform equation~\eqref{eq:u} in polar coordinates into~\eqref{eq1440mon2}.
In light of~\eqref{eq2038sat} and \eqref{eq1320thu}, the solution to the Dirichlet problem \eqref{eq:dp_in_disk} can therefore be expressed as
\[
u(r,\theta)=\sum_{n=0}^\infty r^n a_n \phi_n^{(+)}+ \sum_{n=0}^\infty r^n  b_n \phi_n^{(-)},\qquad b_0=0.
\]
It is straightforward to verify that the functions
\[
u^{(\pm)}_n(r,\theta):=r^n \phi_n^{(\pm)}(\theta),\qquad n =0,1,2,\ldots,
\]
are weak solutions of \eqref{eq:u}.
By using \eqref{eq1327thu} and \eqref{eq1328thu}, we obtain
\begin{multline}			\label{eq1003sat}
u(r, \theta)=
\sum_{n=0}^\infty r^n a_n \left((-i)^{n}(-k \mathsf{s}) e^{-i n\theta}+i^{n}(k \mathsf{s}+2i)e^{i n\theta} \right) \\
+ \sum_{n=0}^\infty r^n b_n \left( i^n (-k \mathsf{s}) e^{i n\theta}+(-i)^n(k \mathsf{s} -2i)e^{-i n\theta} \right).
\end{multline}

\subsection{Representation of the solution}
For $0 < r<1$, introduce the operator $\mathsf{P_r}$ that multiplies $r^n$ at Fourier mode $n$, that is, for $f\in L^2(\partial\mathbb{D})$, we define
\[
\mathsf{P_r} f(\theta):=\sum_{n=-\infty}^\infty r^{\abs{n}} \hat f(n) e^{in\theta}.
\]
Note that $\mathsf{P_r}$ coincides with the classical Poisson extension operator
\[
\mathcal{P}_r \ast f(\theta) = \frac{1}{2\pi} \int_{-\pi}^{\pi} \mathcal{P}_r(\theta-\phi) f(\phi)\,d\phi,\qquad 
\mathcal{P}_r(\phi) := \frac{1-r^2}{1-2r\cos\phi + r^2},
\]
where $\mathcal{P}_r(\phi)$ is the Poisson kernel on the unit circle.

We also introduce the multiplication operator by the symbol $\mathsf{s}$,
\[
\mathsf{M_s} f(\theta):= \mathsf{s}(\theta) f(\theta).
\]
With this notation, equation \eqref{eq1003sat} may be rewritten as
\begin{equation}		\label{eq1431thu}
u(r,\theta)= 2i \mathsf{P_r} (\alpha - \opR\beta)(\theta) + k\mathsf{M_s} \mathsf{P_r} [ (\opI-\opR)(\alpha-\beta) ](\theta).
\end{equation}

Our goal in this subsection is to obtain a representation formula for \eqref{eq1431thu} in terms of the given datum $g \in L^2(\partial\mathbb{D})$.
Consider the Hardy projection operators
\[
\opP^{+}f(\theta) = \sum_{n = 0}^\infty \hat{f}(n)e^{in\theta},\qquad \opP^{-}:=\opI-\opP^{+}.
\]
Applying $\opP^{+}$ and $\opP^{-}$ to \eqref{eq_for_g}  gives (recall $b_0=0$)
\begin{equation}	\label{eq1325wed}
\begin{aligned}
\opP^{+} g &= 2i \alpha + k \opP^{+} \mathsf{M_s}(\opI-\opR)(\alpha-\beta),\\
\opP^{-}g &= -2i\opR\beta + k\opP^{-}\mathsf{M_s}(\opI-\opR)(\alpha-\beta),
\end{aligned}
\end{equation}
Using these and \eqref{property_r}, we obtain
\begin{equation}			\label{eq1015sat}
(\opP^{+} + \opR\opP^{-})g= [2i\opI+k(\opP^{+}+\opR\opP^{-})\mathsf{M_s}(\opI-\opR)](\alpha-\beta).
\end{equation}
Define
\[
\mathsf{B} := \opP^{+}+\opR\opP^{-}=\opP^{+} + \opR(\opI-\opP^{+})
\]
and
\begin{equation}			\label{op_t}
\mathsf{T} := \mathsf{B} \mathsf{M_s}(\opI-\opR),
\end{equation}
so that \eqref{eq1015sat} becomes
\[
(2i\opI+k\mathsf{T})(\alpha-\beta) = \mathsf{B} g.
\]
Therefore, we are led to consider the following problem: For $g \in L^2(\partial\mathbb{D})$, find $\psi \in L^2(\partial\mathbb{D})$ such that
\begin{equation}		\label{eq0858mon}
(2i\opI+k\mathsf{T})\psi = \mathsf{B} g.
\end{equation}

Assume for the moment that there exists such a solution $\psi$.
Then we recover $\alpha$ and $\beta$ from $\psi=\alpha-\beta$ as follows.
With $\psi$ satisfying \eqref{eq0858mon}, set
\begin{equation}		\label{eq0923wed}
\tilde{g}:=(\opI - \opR)\psi.
\end{equation}
Then, from \eqref{eq1325wed} and $\psi=\alpha-\beta$, we obtain
\[
\alpha=\frac{1}{2i} \opP^{+}(g - k\mathsf{s}\tilde{g}),\qquad
\beta=\alpha-\psi.
\]
Combining these, we obtain
\begin{equation}		\label{eq1708thu}
\alpha-\opR\beta = \alpha- \opR(\alpha-\psi)=
\frac{1}{2i}(\opI-\opR) \opP^{+}(g - k\mathsf{s}\tilde{g})+\opR \psi.
\end{equation}
Then, by \eqref{eq1431thu} and \eqref{eq1708thu}, we obtain
\begin{equation}			\label{eq1311tue}
u(r,\cdot )=\mathsf{P_r} (\opI-\opR)\opP^{+}(g-k\mathsf{s}\tilde{g})+2i \mathsf{P_r} \opR \psi + k\mathsf{M_s} \mathsf{P_r} \tilde{g}.
\end{equation}
From the equation \eqref{eq0858mon}, and the definitions \eqref{op_t} and \eqref{eq0923wed}, we obtain
\[
2i \psi= \mathsf{B}(g-k\mathsf{s}\tilde{g}).
\]
So, the equation \eqref{eq1311tue} becomes
\begin{equation}			\label{eq1332tue}
u(r,\cdot )=\mathsf{P_r} (\opI-\opR)\opP^{+}(g-k\mathsf{s}\tilde{g})+\mathsf{P_r} \opR \mathsf{B}(g-k\mathsf{s}\tilde{g})+ k\mathsf{M_s} \mathsf{P_r} \tilde{g}.
\end{equation}
Using $\mathsf{B} = \opP^{+} + \opR\opP^{-}$ and $\opP^{-} = \opI-\opP^{+}$, we obtain
\[
 (\opI-\opR)\opP^{+} +\opR \mathsf{B} = \opI.
\]
Hence, the equation \eqref{eq1332tue} is
\[
u(r,\cdot) = \mathsf{P_r}(g-k\mathsf{s}\tilde{g})  +k \mathsf{M_s} \mathsf{P_r}\tilde{g},
\]
and using commutator notation $[\mathsf{P_r}, \mathsf{M_s}]=\mathsf{P_r}\mathsf{M_s} - \mathsf{M_s} \mathsf{P_r}$, we obtain
\begin{equation}		\label{eq1121mon}
u(r,\cdot) = \mathsf{P_r} g -k[\mathsf{P_r}, \mathsf{M_s}] \tilde{g},\qquad \tilde{g}:=(\opI-\opR)\psi,
\end{equation}
where $\psi \in L^2(\partial\mathbb{D})$ is a solution of equation \eqref{eq0858mon}.

\medskip
We will show in Section~\ref{sec:solution} that equation \eqref{eq0858mon} admits a solution $\psi \in L^2(\partial\mathbb{D})$ for every $g \in L^2(\partial\mathbb{D})$ whenever $\abs{k} \neq 1$; see Section~\ref{sec4.5}.
When $\abs{k}=1$, there exists $g \in L^2(\partial\mathbb{D})$ for which no corresponding solution $\psi \in L^2(\partial\mathbb{D})$ exists; see Section~\ref{sec4.6}.

\section{Weak solutions and non-tangential convergence}		\label{sec:weak_sol}

In this section, we show that the function $u$ defined by the formula \eqref{eq1121mon}  indeed solves the Dirichlet problem \eqref{eq:dp_in_disk}.
More precisely, we prove that $u$ belongs to a weighted Sobolev space and, in particular, lies in $H^1_{\mathrm{loc}}(\mathbb{D})$, and \eqref{eq:u} holds in the weak formulation.
We further establish that $u(r,\cdot)\to g$ in $L^2(\partial\mathbb{D})$ as $r\nearrow 1$.
In addition, we prove the $L^2$ non-tangential maximal estimate for $u$ and show that $u$ converges to $g$ non-tangentially a.e. on $\partial\mathbb{D}$.

\subsection{Weighted energy estimates}
Let $g\in L^2(\partial\mathbb{D})$ and assume that there exists a function $\psi \in L^2(\partial\mathbb{D})$ satisfying \eqref{eq0858mon}, and let $\tilde{g}$ and $u$ be defined by the formula \eqref{eq1121mon}.
We establish the weighted energy estimate
\begin{equation}	\label{eq1217tue}
\int_{\mathbb{D}} (1-r)\,\abs{\nabla u(r,\theta)}^2\,r\,dr\,d\theta \le C \left(\norm{g}_{L^2(\partial\mathbb{D})}^2+ k^2\,\norm{\tilde{g}}_{L^2(\partial\mathbb{D})}^2\right).
\end{equation}
In particular, this bound implies that
$u\in H^1_{\mathrm{loc}}(\mathbb{D})$.
Indeed, this follows from combining \eqref{eq1217tue} with the auxiliary estimate
\begin{equation}		\label{eq1632tue}
\sup_{0<r<1}\,\norm{u(r,\cdot)}_{L^2(\partial\mathbb{D})}
\le C \left(\norm{g}_{L^2(\partial\mathbb{D})}+ \abs{k}\, \norm{\tilde{g}}_{L^2(\partial\mathbb{D})}\right),
\end{equation}
which is immediate from \eqref{eq1121mon}; see Section~\ref{sec3.4}.
To prove \eqref{eq1217tue}, define
\begin{equation}			\label{eq:uvw}
U(r,\theta):=\mathcal{P}_r\ast g(\theta),\quad
V(r,\theta):=\mathcal{P}_r\ast(\mathsf{s}\tilde{g})(\theta),\quad
W(r,\theta):=\mathcal{P}_r\ast \tilde{g}(\theta).
\end{equation}
With this notation, we may write
\begin{equation}		\label{eq2119tue}
u=U-kV+k\mathsf{s}W.
\end{equation}

Let
\[
\mathbb{D}_+:=\mathbb{D}\cap\{x>0\},\qquad
\mathbb{D}_-:=\mathbb{D}\cap\{x<0\},
\]
and denote the interface by
\[
\Gamma:=\mathbb{D}\cap\{x=0\}.
\]

In each of the half-disks $\mathbb{D}_\pm$, the functions $U$, $V$, and $\mathsf{s}W$ are harmonic and hence smooth. A possible jump across $\Gamma$ may arise from the discontinuity of $\mathsf{s}$, but the following structural lemma prevents this.

\begin{lemma}\label{lem2117mon}
Let $\tilde{h}=(\opI-\opR) h$ with $h\in L^2(\partial\mathbb{D})$.
Then, for all $0<r<1$,
\[
\mathcal{P}_r\ast \tilde{h} \left(\pm\tfrac{\pi}{2}\right)=0.
\]
Consequently, $\mathcal{P}_r\ast \tilde{h}$ vanishes on $\Gamma$.
\end{lemma}

\begin{proof}
Since $\opR(\opI-\opR)=-(\opI-\opR)$, we have
$\opR \tilde{h}=-\tilde{h}$.
The Poisson extension operator commutes with $\opR$, so $\mathcal{P}_r\ast \tilde{h}$ also satisfies
$\opR(\mathcal{P}_r\ast \tilde{h})=-\mathcal{P}_r\ast \tilde{h}$.
Hence $\mathcal{P}_r\ast \tilde{h}$ vanishes at $\theta=\pm\frac{\pi}{2}$, which correspond exactly to points on $\Gamma$.
\end{proof}

By the classical Littlewood--Paley theory, the Poisson extension $U$ satisfies
\[
\int_{\mathbb{D}} (1-r)\abs{\nabla U(r,\theta)}^2\,r\,dr\,d\theta
\le C\,\norm{g}_{L^2(\partial\mathbb{D})}^2,
\]
where $C>0$ is an absolute constant. Similarly,
\[
\int_{\mathbb{D}} (1-r)\abs{\nabla V(r,\theta)}^2\,r\,dr\,d\theta
\le C\,\norm{\mathsf{s}\tilde{g}}_{L^2(\partial\mathbb{D})}^2= C\,\norm{\tilde{g}}_{L^2(\partial\mathbb{D})}^2,
\]
where the final inequality follows from the fact that $\mathsf{s}^2=1$ a.e.

Next, take $h=\psi$ in Lemma~\ref{lem2117mon}.
Then $W=\mathcal{P}_r\ast \tilde{g}$ vanishes on $\Gamma$.
Since $\mathsf{s}=\pm1$ on $\mathbb{D}_\pm$, a standard gluing argument across $\Gamma$ yields
\begin{equation}			\label{eq2022mon}
\nabla(\mathsf{s}W)= \mathsf{s} \nabla W \quad \text{a.e. in } \mathbb{D}.
\end{equation}
Consequently,
\[
\int_{\mathbb{D}} (1-r)\abs{\nabla(\mathsf{s}W)(r,\theta)}^2\,r\,dr\,d\theta
=\int_{\mathbb{D}} (1-r)\abs{\nabla W(r,\theta)}^2\,r\,dr\,d\theta
\le C\,\norm{\tilde{g}}_{L^2(\partial\mathbb{D})}^2.
\]

Combining the estimates for $U$, $V$, and $W$, together with the decomposition
$u=U-kV+k\mathsf{s}W$, we obtain the weighted energy estimate \eqref{eq1217tue}.

\subsection{Verification of the transmission condition}	\label{sec3.2}
We show that $u$ is a weak solution of \eqref{eq:u}.
Since the weighted energy estimates give $u\in H^1_{\mathrm{loc}}(\mathbb{D})$ and $u$ is harmonic in each half-disk $\mathbb{D}_\pm$, it remains only to verify continuity of the conormal flux across $\Gamma$.
In Cartesian coordinates, the transmission condition is
\begin{equation}\label{eq1539tue}
[(\mathbf A\nabla u)\cdot \vec n]_\Gamma
=[u_x+k\mathsf{s}u_y]_\Gamma=0,
\end{equation}
where $[\,\cdot\,]_\Gamma$ denotes the jump across $\Gamma$.
For $U=\mathsf{P_r} g$, we have
\[
[U_x+k\mathsf{s}U_y]_\Gamma=2k U_y=2k\,\partial_y(\mathsf{P_r} g).
\]
Similarly, for $V=\mathsf{P_r}(\mathsf{s}\tilde{g})$,
\[
[V_x+k\mathsf{s}V_y]_\Gamma=2k V_y=2k\,\partial_y(\mathsf{P_r}(\mathsf{s}\tilde{g})).
\]
For $\mathsf{s}W=\mathsf{s}\mathsf{P_r} \tilde{g}$, note that $\partial_x(\mathsf{s}W)=\mathsf{s}W_x$ and $\partial_y(\mathsf{s}W)=\mathsf{s}W_y$ on each side of $\Gamma$. Hence,
\[
[(\mathsf{s}W)_x+k\mathsf{s}(\mathsf{s}W)_y]_\Gamma=2\mathsf{s} W_x=2\,\partial_x(\mathsf{P_r} \tilde{g}).
\]

Combining these expressions and using \eqref{eq2119tue}, we obtain
\[
[u_x+k\mathsf{s}u_y]_\Gamma
=2k\,\partial_y(\mathsf{P_r} g)
-2k^2\,\partial_y \left(\mathsf{P_r}(\mathsf{s}\tilde{g})\right)
+2k\,\partial_x(\mathsf{P_r} \tilde{g}).
\]

\medskip
Recall that the Hilbert transform on $\partial\mathbb{D}$ is defined by
\begin{equation}			\label{eq:hilbert_t}
\opH f(\theta)=\mathrm{p.v.}\,\frac{1}{2\pi}\int_{\partial\mathbb{D}}
f(\phi)\cot\frac{\theta-\phi}{2}\,d\phi.
\end{equation}
It is well known that $\opH$ satisfies the following property.
\[
\opH e_n = (-i \sgn n)\, e_{n},\qquad e_n(\theta):=e^{i n \theta},\qquad \sgn 0 =0.
\]

Since $\mathsf{P_r} \tilde{g}+i\mathsf{P_r}(\opH \tilde{g})$ is analytic in $\mathbb{D}$, the Cauchy--Riemann equations imply
\[
\partial_x(\mathsf{P_r} \tilde{g})=\partial_y(\mathsf{P_r}(\opH \tilde{g})).
\]
Substituting this identity into the flux expression yields
\[
[u_x+k\mathsf{s} u_y]_\Gamma =2k\,\partial_y(\mathsf{P_r}(g-k\mathsf{s}\tilde{g}+\opH \tilde{g})).
\]
Thus, it suffices to show that
\[
\partial_y\mathsf{P_r} \varphi=0 \quad\text{on } \Gamma, \qquad \text{where }\;\varphi :=g-k\mathsf{s} \tilde{g}+\opH \tilde{g}.
\]

\begin{lemma}\label{lem1431tue}
Let $\varphi \in L^2(\partial \mathbb{D})$ satisfy $(\opI - \opR) \opH \varphi =0$.
Then $\partial_y \mathsf{P_r} \varphi =0$ on $\Gamma$.
\end{lemma}

\begin{proof}
It follows from $\opR\opH=-\opH\opR$ that
\[
\opH(\varphi+\opR \varphi)=0,
\]
which implies $\varphi+\opR \varphi=c$ for some constant $c$, since $\ker \opH=\operatorname{span}\{1\}$.

Because $\mathsf{P_r}$ commutes with $\opR$, the function $v=\mathsf{P_r} \varphi$ satisfies
\[
v(x,y)+v(-x,y)=c
\quad\text{in } \mathbb{D}.
\]
Restricting to $\Gamma=\mathbb{D}\cap\{x=0\}$ yields $v(0,y)=c/2$, and hence $\partial_y v=0$ on $\Gamma$.
\end{proof}

By Lemma \ref{lem1431tue}, it only remains to show that $(\opI-\opR) \opH \varphi=0$.
Since $\tilde{g}=(\opI-\opR)\psi$ and $\psi$ satisfies \eqref{eq0858mon}, we have
\[
2 i \psi+ k \mathsf{B}(\mathsf{s}\tilde{g}) = \mathsf{B}g.
\]
Applying $\opI-\opR$ to the equation yields
\begin{equation}	\label{eq1508tue}
2i \tilde{g}+k(\opI-\opR)\mathsf{B}(\mathsf{s}\tilde{g})=(\opI-\opR)\mathsf{B} g.
\end{equation}

\begin{lemma}\label{lem:bandh}
The operators $\mathsf{B}$ and $\opH$ satisfy the identity
\[
(\opI-\opR)\mathsf{B}=i(\opI-\opR)\opH.
\]
\end{lemma}

\begin{proof}
We first note that
\[
(\opI-\opR)\mathsf{B}=(\opI-\opR)(2\opP^{+}-\opI),
\]
which follows from $\mathsf{B}=\opP^{+}+\opR(\opI-\opP^{+})$ together with $\opR^2=\opI$.
Next, note that
\[
(2\opP^{+}-\opI)f=i\opH f+\hat f(0).
\]
Since $\opR1=1$, the constant term $\hat f(0)$ is annihilated by $\opI-\opR$. Therefore,
\[
(\opI-\opR)(2\opP^{+}-\opI)=i(\opI-\opR)\opH,
\]
which proves the lemma.
\end{proof}

By Lemma \ref{lem:bandh}, we obtain from \eqref{eq1508tue}
\begin{equation}	\label{eq1531wed}
2\tilde{g}+k(\opI-\opR)\opH(\mathsf{s}\tilde{g})=(\opI-\opR)\opH g.
\end{equation}
On the other hand, the function $\varphi=g-k \mathsf{s}\tilde{g}+\opH \tilde{g}$ satisfies
\begin{align*}
(\opI-\opR)\opH\varphi &=(\opI-\opR)\opH g-k(\opI-\opR)\opH(\mathsf{s}\tilde{g}) +(\opI-\opR)\opH^2 \tilde{g}\\
&=(\opI-\opR)\opH g-k(\opI-\opR)\opH(\mathsf{s}\tilde{g})-2\tilde{g},
\end{align*}
where we used the fact that $\tilde{g}=(\opI-\opR)\psi$ has zero average, and hence $\opH^2 \tilde g = -\tilde g$.
Indeed
\[
(\opI-\opR)\opH^2 \tilde{g}=-(\opI-\opR)\tilde{g}=-(\opI-\opR)^2 \psi=-2(\opI-\opR)\psi=-2\tilde{g}.
\]
Combining this with \eqref{eq1531wed} and Lemma~\ref{lem1431tue}, we find that $u$ satisfies the transmission condition \eqref{eq1539tue} and is therefore a weak solution of \eqref{eq:u}.

\subsection{Boundary convergence in $L^2$}		\label{sec3.4}
We show that
\[
\lim_{r \nearrow 1}\,\norm{u(r,\cdot)-g}_{L^2(\partial\mathbb{D})}=0.
\]
Recall the decomposition \eqref{eq2119tue},
\begin{equation}		\label{eq1138wed}
u=U-kV+k\mathsf{s}W =\mathsf{P_r} g -k \mathsf{P_r}(\mathsf{s}\tilde{g}) +k\mathsf{s} \mathsf{P_r} \tilde{g}.
\end{equation}

By the triangle inequality,
\[
\norm{u(r,\cdot)-g}_{L^2(\partial\mathbb{D})}
\le
\norm{\mathsf{P_r} g-g}_{L^2(\partial\mathbb{D})}
+\abs{k}\,\norm{\mathsf{P_r}(\mathsf{s}\tilde{g})-\mathsf{s}\mathsf{P_r} \tilde{g}}_{L^2(\partial\mathbb{D})}.
\]
Adding and subtracting $\mathsf{s}\tilde{g}$ and using $\abs{\mathsf{s}}=1$ a.e., we obtain
\begin{align*}
\norm{\mathsf{P_r}(\mathsf{s}\tilde{g})-\mathsf{s}\mathsf{P_r} \tilde{g}}_{L^2(\partial\mathbb{D})}
&\le \norm{\mathsf{P_r}(\mathsf{s}\tilde{g})-\mathsf{s}\tilde{g}}_{L^2(\partial\mathbb{D})} +\norm{\mathsf{s}(\tilde{g}-\mathsf{P_r} \tilde{g})}_{L^2(\partial\mathbb{D})} \\
&= \norm{\mathsf{P_r}(\mathsf{s}\tilde{g})-\mathsf{s}\tilde{g}}_{L^2(\partial\mathbb{D})} +\norm{\tilde{g}-\mathsf{P_r} \tilde{g}}_{L^2(\partial\mathbb{D})}.
\end{align*}
By the standard $L^2$ convergence of the Poisson integral, we obtain
\[
\lim_{r\nearrow 1}\, \norm{u(r,\cdot)-g}_{L^2(\partial\mathbb{D})}
\le 0+ \abs{k}\cdot (0 + 0) = 0.
\]
Hence $u(r,\cdot)\to g$ in $L^2(\partial\mathbb{D})$.

Finally, we note that \eqref{eq1632tue} follows directly from \eqref{eq1138wed}, $\abs{\mathsf{s}}=1$ a.e., and the contraction property of the operator $\mathsf{P_r}$ on $L^2(\partial\mathbb{D})$.

\subsection{Non-tangential convergence}
For $\phi \in \partial\mathbb{D}$, we define the non-tangential maximal function
\begin{equation}		\label{eq0854wed}
\mathcal{N}_*u(\phi):= \sup \left\{\abs{u(r,\theta)}:  re^{i\theta} \in \mathcal{A}_\beta(\phi) \right\},
\end{equation}
where the non-tangential approach region $\mathcal{A}_\beta(\phi)$ is given by
\begin{equation}		\label{eq0856wed}
\mathcal{A}_\beta(\phi):= \left\{ z \in \mathbb{C}:  \abs{z-e^{i\phi}} < \beta (1-\abs{z}) \right\}
\end{equation}
for some fixed aperture constant $\beta>1$.

We also define the Poisson non-tangential maximal function by
\[
\mathcal{P}_{\ast} g(\phi):= \sup \left\{\,\abs{\mathcal{P}_r\ast g(\theta)}:  re^{i\theta} \in \mathcal{A}_\beta(\phi) \right\}.
\]
It is well known (see, e.g., Stein \cite{Stein1970}) that
\begin{equation}		\label{eq0800wed}
\norm{\mathcal{P}_{\ast}g}_{L^2(\partial\mathbb{D})} \le C \norm{g}_{L^2(\partial\mathbb{D})},
\end{equation}       
where $C=C(\beta)$.
Recalling the decomposition \eqref{eq2119tue},
\begin{equation}		\label{eq0808wed}
u(r,\theta)=\mathcal{P}_r \ast g(\theta) -k \mathcal{P}_r \ast (\mathsf{s} \tilde{g})(\theta) +k\mathsf{s}(\theta)\, \mathcal{P}_r \ast \tilde{g}(\theta),
\end{equation}
we obtain the pointwise bound
\begin{equation}		\label{eq0803wed}
\abs{u(r,\theta)} \le \abs{\mathcal{P}_r \ast g(\theta)}+ \abs{k}\, \abs{\mathcal{P}_r \ast (s \tilde{g})(\theta)}+ \abs{k}\,\abs{\mathcal{P}_r \ast \tilde{g}(\theta)}.
\end{equation}
Taking the supremum in \eqref{eq0803wed} over $re^{i\theta} \in \mathcal{A}_\beta(\phi)$ yields
\begin{equation}		\label{eq0839wed}
\mathcal{N}_*u(\phi) \le \mathcal{P}_{\ast} g(\phi)+ \abs{k} \mathcal{P}_{\ast} (\mathsf{s} \tilde{g})(\phi)+ \abs{k}\,\mathcal{P}_{\ast} \tilde{g}(\phi).
\end{equation}
Therefore, by \eqref{eq0800wed} and $\abs{\mathsf{s}}=1$ a.e.,
\begin{equation*}
\norm{\mathcal{N}_*u}_{L^2(\partial\mathbb{D})} \le C\left(\norm{g}_{L^2(\partial\mathbb{D})} + \abs{k}\, \norm{\tilde{g}}_{L^2(\partial\mathbb{D})} \right),\qquad C=C(\beta)>0.
\end{equation*}

Finally, each of the Poisson integrals $\mathcal{P}_r \ast g$, $\mathcal{P}_r \ast (\mathsf{s} \tilde{g})$, and $\mathcal{P}_r \ast \tilde{g}$ in \eqref{eq0808wed} has a non-tangential limit almost everywhere, equal to its boundary datum. 
Since $\mathsf{s}$ is a bounded multiplier, it follows from \eqref{eq0808wed} that
\begin{equation}		\label{eq0841wed}
\lim_{\substack{re^{i\theta} \to e^{i\phi}\\ re^{i\theta} \in \mathcal{A}_\beta(\phi)}} u(r,\theta)  = g(\phi) -k (\mathsf{s}\tilde{g})(\phi) +k \mathsf{s}(\phi)\tilde{g}(\phi)=g(\phi),\quad \text{for a.e. }\, \phi \in \partial\mathbb{D}.
\end{equation}
This establishes the non-tangential convergence of $u$ to $g$ a.e. on $\partial \mathbb{D}$.

\subsection{$H^1$-regularity of the solution and the maximum principle}
We define the Sobolev spaces
\begin{equation}		\label{eq1732mon}
H^\alpha(\partial\mathbb{D})=\left\{f \in L^2(\partial\mathbb{D}): \sum_{n=-\infty}^\infty (1+\abs{n}^2)^\alpha \,\abs{\hat f(n)}^2 <\infty \right\},\qquad \alpha \in \mathbb{R}.
\end{equation}
Hereafter, we will use the following notational convention:
$g \in H^{1/2+}(\partial\mathbb{D})$ means that $g \in H^{1/2+\epsilon}(\partial\mathbb{D})$ for some $\epsilon>0$; that is,
\[
H^{1/2+}(\partial\mathbb{D}):=\bigcup_{\epsilon>0} H^{1/2+\epsilon}(\partial\mathbb{D}).
\]

The following fact is well known.
If $f \in H^{1/2}(\partial\mathbb{D})$, then its Poisson extension $v(r,\theta) :=\mathcal{P}_r \ast f(\theta)$ satisfies
\begin{equation}		\label{eq1025tue}
\norm{\nabla v}_{L^2(\mathbb{D})}^2 \asymp \int_{\mathbb D} \abs{\nabla v(r,\theta)}^2\,r dr d\theta \asymp \sum_{n=-\infty}^\infty \abs{n}\,\abs{\hat f(n)}^2 \lesssim  \norm{f}_{H^{1/2}(\partial\mathbb{D})}^2.
\end{equation}

The following lemma is a consequence of a standard gluing argument.
\begin{lemma}		\label{lem1051tue}
Let $\frac12 <\alpha <\frac32$ and let $f\in H^\alpha(\partial\mathbb{D})$ satisfy
\[
f(\tfrac{\pi}{2})=f(-\tfrac{\pi}{2})=0.
\]
Then $\mathsf{M_s} f = \mathsf{s}f \in H^\alpha(\partial\mathbb{D})$ and
\[
\norm{\mathsf{M_s} f}_{H^\alpha(\partial\mathbb{D})} \le C_\alpha \norm{f}_{H^\alpha(\partial\mathbb{D})}.
\]
In particular, when $\alpha=1$, one may take $C_1=1$, in which case equality holds with the $H^1$ norm defined by $\norm{f}_{H^1}^2:=\norm{f}_{L^2}^2+\norm{f'}_{L^2}^2$.
\end{lemma}
\begin{proof}
See Lions and Magenes \cite[Theorem~11.4]{LM1972} for the standard gluing argument.
When $\alpha=1$, we compute
\[
\norm{\mathsf{M_s} f}_{H^1(\partial\mathbb{D})}^2=\norm{\mathsf{s} f}_{L^2(\partial\mathbb{D})}^2+\norm{(\mathsf{s} f)'}_{L^2(\partial\mathbb{D})}^2=\norm{f}_{L^2(\partial\mathbb{D})}^2+\norm{f'}_{L^2(\partial\mathbb{D})}^2= \norm{f}_{H^1(\partial\mathbb{D})}^2,
\]
where we used $f(\pm \pi/2)=0$ to obtain $(\mathsf{s}f)'=\mathsf{s}f'$ a.e.
Thus, $C_1=1$.
\end{proof}

Let $g\in H^{1/2+}(\partial\mathbb{D})$ and assume that there exists a solution $\psi$ of the equation \eqref{eq0858mon} such that $\psi \in H^{1/2+}(\partial\mathbb{D})$.
Under this assumption, $\tilde{g}=(\opI-\opR)\psi$ satisfies the hypothesis of Lemma \ref{lem1051tue}.
Therefore, by Lemma \ref{lem1051tue}, we see that $\mathsf{s}\tilde{g} \in H^{1/2+}(\partial\mathbb{D})$.

Decompose $u=U-kV+k\mathsf{s}W$ with $U$, $V$, and $W$ defined by \eqref{eq:uvw}.
By \eqref{eq1025tue}, we find that $U$, $V$, and $W$ all belong to $H^1(\mathbb{D})$.
Moreover, $\mathsf{s}W \in H^1(\mathbb{D})$ and satisfies $\nabla(\mathsf{s}W)=\mathsf{s} \nabla W$ as in \eqref{eq2022mon}, since $W$ vanishes on $\Gamma$ by Lemma~\ref{lem2117mon}.

Therefore, $u \in H^1(\mathbb{D})$.
If $g$ and $\tilde{g}$ are both real-valued, then $u$ is real-valued.
Hence, by the classical De Giorgi--Moser--Nash theory, it follows that $u$ is the unique weak solution in $H^1(\mathbb{D})$ of the Dirichlet problem \eqref{eq:dp_in_disk}.
Moreover, $u$ extends continuously to $\overline{\mathbb{D}}$, and the weak maximum principle holds; see \cite[Chapter~8]{GT}.

\subsection{Summary}
We summarize the results established in this section.
\begin{theorem}		\label{thm1}
Let $g\in L^2(\partial\mathbb{D})$ and $k \in \mathbb{R}$.
Suppose there exists $\psi \in L^2(\partial\mathbb{D})$ satisfying \eqref{eq0858mon}, and define $\tilde{g}$ and $u$ by the formula \eqref{eq1121mon}.
Then $u$ enjoys the following properties:
\begin{enumerate}
\item \emph{(Energy estimates).} 
The function $u$ belongs to $H^1_{\mathrm{loc}}(\mathbb{D})$ and satisfies the weighted energy estimate
\[
\int_{\mathbb{D}} (1-r)\,\abs{\nabla u(r,\theta)}^2\,r\,dr\,d\theta \le C \left(\norm{g}_{L^2(\partial\mathbb{D})}^2+ k^2\,\norm{\tilde{g}}_{L^2(\partial\mathbb{D})}^2\right),
\]
where $C$ is an absolute constant.
\item \emph{(Weak solution).} 
The function $u$ is a weak solution of \eqref{eq:u}.
In particular, $u$ is harmonic in $\mathbb{D}_\pm$ and satisfies the conormal flux transmission condition across the interface $\Gamma$,
\[
[u_x+k\mathsf{s}u_y]_\Gamma=0.
\]
\item \emph{(Boundary convergence in $L^2$).}
The function $u$ attains the boundary value $g$ in the $L^2$-sense, namely,
\[
\lim_{r \nearrow 1}\,\norm{u(r,\cdot)-g}_{L^2(\partial\mathbb{D})}=0.
\]
\item \emph{(Non-tangential maximal estimate and a.e. convergence).}
The non-tangential maximal function $\mathcal{N}_*u$, defined in \eqref{eq0854wed}, satisfies
\[
\norm{\mathcal{N}_*u}_{L^2(\partial\mathbb{D})} \le C\left( \norm{g}_{L^2(\partial\mathbb{D})} + \abs{k}\, \norm{\tilde{g}}_{L^2(\partial\mathbb{D})}  \right),
\]
where $C=C(\beta)$ and $\beta$ is the aperture constant appearing in \eqref{eq0856wed}.
Moreover, $u$ converges to $g$ non-tangentially a.e. on $\partial\mathbb{D}$.
\item \emph{(Maximum principle).}
Suppose further that $g$ and $\tilde{g}$ both belong to $H^{1/2+}(\partial\mathbb{D})$ and are real-valued.
Then $u \in H^1(\mathbb{D})$ and $u$ satisfies the weak maximum principle.
\end{enumerate}
\end{theorem}

\section{Solvability of the boundary equation \eqref{eq0858mon}} \label{sec:solution}

In this section, we study the solvability of the boundary equation \eqref{eq0858mon}.
We equip $L^2=L^2(\partial\mathbb{D})$ with the inner product
\[
\ip{f,g}=\frac{1}{2\pi} \int_{-\pi}^{\pi} f(\theta)\overline{g(\theta)}\,d\theta.
\]
Let $H_{\rm Hardy}^2$ denote the classical Hardy space
\[
H_{\rm Hardy}^2=\{ f \in L^2(\partial\mathbb{D}):  \hat f(n)=0, \;  \forall n<0 \}.
\]
Here and below, we use the notation
\[
\hat f(n)= \ip{f, e_n},\qquad e_n=e_n(\theta)=e^{in\theta}.
\]

\subsection{Basic setup}

Recall that the operator $\mathsf{T}$ is defined by
\[
\mathsf T:=\mathsf B \mathsf{M_s}(\opI-\opR),
\]
where $\mathsf{M_s}$ denotes multiplication by $\mathsf{s}(\theta)=\sgn(\cos\theta)$ and $\mathsf{B}$ is given by
\[
\mathsf B:=\opP^{+} + \opR\opP^{-},
\]
with $\opP^{+}$ and $\opP^{-}$ the Hardy projections,
\[
\opP^{+}f=\sum_{n=0}^\infty \hat f(n)e_n,\qquad \opP^{-}=\opI-\opP^{+}.
\]
The operator $\opR$ is the unitary reflection
\[
(\opR f)(\theta)=f(\pi-\theta),
\]
which equivalently satisfies
\[
\opR e_n = (-1)^n e_{-n}.
\]
We recall that the Hilbert transform $\opH$ satisfies
\begin{equation}		\label{eq1703fri}
\opH f= \sum_{n=-\infty}^\infty \hat f(n) (-i \sgn n) e_n,\qquad \sgn 0=0.
\end{equation}

\subsection{Analysis of the operator $\mathsf{B}$}
Recalling that $\opR^2=\opI$, we define the projection operators
\begin{equation}	\label{eq1112sun}
\mathsf{\Pi}_+:=\tfrac12 (\opI+\opR),\qquad \mathsf{\Pi}_-:=\tfrac12 (\opI-\opR),
\end{equation}
and the symmetric and antisymmetric subspaces
\[
L^2_+:=\ran \mathsf{\Pi}_+=\ker(\opI-\opR),\qquad L^2_-:=\ran \mathsf{\Pi}_-=\ker(\opI+\opR),
\]
so that $L^2 = L^2_- \oplus L^2_+$ is an orthogonal decomposition.

\begin{lemma}	\label{lem:B_isometry}
We have
\begin{equation}			\label{eq0409mon}
\mathsf{B}f = f+i \opH f\quad \text{for }\;f \in L^2_+ \qquad\text{and}\qquad \mathsf{B}f =0\quad\text{for }\; f \in L^2_-.
\end{equation}
Moreover, $\mathsf{B}\vert_{L^2_+}: L^2_+ \to H_{\rm Hardy}^2$ is a bijection with a bounded inverse
\begin{equation}		\label{eq1351mon}
\mathsf{B}\vert_{L^2_+}^{-1}=\mathsf{\Pi}_+ \vert_{H_{\rm Hardy}^2}.
\end{equation}
\end{lemma}
\begin{proof}
From the definition of $\opP_+$ and $\opP_-$, we have
\[
\opP^+ f =\frac12 \left(f+ i \opH f +\hat f(0) \right),\quad 
\opP^- f = \frac12 \left(f- i \opH f -\hat f(0) \right),\qquad f \in L^2.
\]
Therefore, we have
\begin{equation}			\label{eq1324mon}
\mathsf{B} f =\frac12(f+ \opR f)+\frac12 i (\opH f + \opH \opR f).
\end{equation}
Here, we used
\[
\opR 1=1,\qquad \opR \opH = - \opH \opR.
\]
Since we have $f=\opR f$ for $f \in L^2_+$ and $f=-\opR f$ for $f \in L^2_-$, we obtain \eqref{eq0409mon} from \eqref{eq1324mon}.

\smallskip
Next, it is clear that $\ran \mathsf{B}=H_{\rm Hardy}^2$.
Suppose $f \in L^2_+$.
Then $\opH f \in L^2_-$ since
\[
\opR \opH f =- \opH \opR f = -\opH f.
\]
Hence, if $f \in L^2_+$ and $\mathsf{B} f=0$, then it follows from $f+i \opH f=0$ that $f = - i \opH f \in L^2_-$.
Since $L^2_+ \cap L^2_- = \{0\}$, we must have $f=0$.
This shows that $\mathsf{B}\vert_{L^2_+}$ is one-to-one.

It remains to show $\mathsf{B}\vert_{L^2_+}$ is onto $H_{\rm Hardy}^2$ and to establish \eqref{eq1351mon}.
Note that
\[
\opP^+ h=h,\quad \opP^-h=0, \quad \opP^+ \opR h= \hat h(0), \quad \opR \opP^-\opR h =h -\hat h(0),\qquad \forall h \in H_{\rm Hardy}^2.
\]
Therefore, we have $\mathsf{B} h =h $ and $\mathsf{B} \opR h = h$ for $h \in H_{\rm Hardy}^2$.
Consequently, $\mathsf{B} \mathsf{\Pi}_+h=h$ for $h \in H_{\rm Hardy}^2$.
It is clear that $\mathsf{\Pi}_+h \in L^2_+$.
\end{proof}

\subsection{Reduction to  $L^2_+$}

Suppose $\psi \in L^2$ satisfies \eqref{eq0858mon}, that is,
\begin{equation}	\label{eq1555mon}		
2i \psi+ k \mathsf{B} \mathsf{M_s}(\opI-\opR)\psi = \mathsf{B} g.
\end{equation}
Then $\psi \in H^2_{\rm Hardy}$ as $2i\psi \in \ran \mathsf{B}= H^2_{\rm Hardy}$.
Therefore, by Lemma \ref{lem:B_isometry}, we may write
\[
\psi=\mathsf{B} y=y+i\mathsf{H}y,\qquad y \in L^2_+.
\]
Then, using $y=\opR y$ and $\opR \opH = -\opH \opR$, we have
\begin{equation}		\label{eq1825fri}
(\opI-\opR)\psi=(\opI-\opR)(y+i\opH y)=2i\opH y.
\end{equation}
Therefore, $\psi \in L^2$ is a solution of \eqref{eq1555mon} if and only if $y \in L^2_+$ is a solution of
\[
2i( \mathsf{B}y+k\mathsf{B} \mathsf{M_s} \opH y)=\mathsf{B}g.
\]
By applying $\mathsf{B}\vert_{L^2_+}^{-1}$ to the equation above, we find that there is a one-to-one correspondence between $\psi \in H^2_{\rm Hardy}$ solving \eqref{eq1555mon} and $y \in L^2_+$ solving
\begin{equation}		\label{eq1601mon}
y+k \mathsf{M_s} \opH y = \frac{1}{2i}\,\mathsf{\Pi}_+g,
\end{equation}
where we used Lemma \ref{lem:B_isometry} to get
\[
\mathsf{B}g=\mathsf{B}(\mathsf{\Pi}_+g + \mathsf{\Pi}_- g)= \mathsf{B} \mathsf{\Pi}_+ g.
\]
Let us define the operator $\mathsf{C}$ on $L^2_+$ as follows:
\[
\mathsf{C}= \mathsf{M_s} \opH \vert_{L^2_+}.
\]
This is well-defined since
\begin{equation}	\label{eq1609mon}
\mathsf{M_s}: L^2_- \to L^2_+,\qquad \opH: L^2_+ \to L^2_-.
\end{equation}

\subsection{Spectrum of the operator $\mathsf{C}$}

\begin{lemma}			\label{lem:C_spectrum}
The spectrum of $\mathsf{C}$ is the set $\{ z \in \mathbb{C}: \abs{z} \le 1\}$, that is, $\sigma(\mathsf{C})= \overline{\mathbb{D}}$.
\end{lemma}
\begin{proof}
Let $f \in L^2_+$.
By the symmetry, we may restrict to $\theta \in (-\frac{\pi}{2}, \frac{\pi}{2})$, where $\mathsf{s}=1$.
Using \eqref{eq:hilbert_t}, together with the identity $\cot(\theta-\frac{\pi}{2})=-\tan \theta$, it follows that
\[
\mathsf{C} f(\theta)= \frac{1}{2\pi}\,\mathrm{p.v.}\int_{-\pi/2}^{\pi/2}
f(\phi)\left(\cot\frac{\theta-\phi}{2}-\tan\frac{\theta+\phi}{2}\right)\,d\phi,\qquad \theta \in (-\tfrac{\pi}{2},\tfrac{\pi}{2}).
\]
Using the elementary identity
\begin{equation}		\label{eq1225thu}
\cot\left(\frac{\theta-\phi}{2}\right) - \tan\left(\frac{\theta+\phi}{2}\right) = \frac{2\cos\theta}{\sin\theta-\sin\phi},
\end{equation}
we obtain the representation
\begin{equation}		\label{eq0854thu}
\mathsf{C} f(\theta)=\frac{\cos\theta}{\pi}\,\mathrm{p.v.} \int_{-\pi/2}^{\pi/2}\frac{f(\phi)}{\sin\theta-\sin \phi}\,d\phi,\qquad -\frac{\pi}{2}<\theta< \frac{\pi}{2}.
\end{equation}

We claim that the function $f_\mu$ defined by
\[
f_\mu(\theta)=\left( \frac{1+\sin\theta}{1-\sin\theta}\right)^{\mu/2},\qquad \Re \mu \in (-\tfrac12,\tfrac12)
\]
is an eigenfunction of $\mathsf{C}$ corresponding to the eigenvalue $\lambda=-\tan (\pi \mu/2)$.

It is straightforward to verify that $f_\mu \in L^2$ when $\abs{\Re \mu} <\frac12$.
Indeed, note that
\begin{align*}
1-\sin \theta &\asymp \tfrac12 (\theta-\pi/2)^2,\qquad 1+\sin\theta \asymp 2 \qquad\text{near }\;\pi/2,\\
1+\sin\theta & \asymp \tfrac12(\theta+\pi/2)^2 ,\qquad 1-\sin \theta \asymp 2 \qquad\text{near }\;-\pi/2.
\end{align*}
Moreover, since $f_\mu(\pi-\theta)=f_\mu(\theta)$, we have $f_\mu \in L^2_+$.

\medskip
To compute $\mathsf{C} f_\mu$, we use the change of variables,
\[
\theta=2u-\tfrac{\pi}{2},\quad x=\tan u,\quad \phi=2v-\tfrac{\pi}{2},\quad y=\tan v,
\]
so that
\[
\sin \theta= -\cos 2u= -\frac{1-x^2}{1+x^2},\quad \cos\theta=\sin 2u= \frac{2x}{1+x^2},\quad \frac{1+\sin\theta}{1-\sin\theta}=x^2,
\]
and similarly
\[
\sin \phi= -\frac{1-y^2}{1+y^2},\quad \cos\phi= \frac{2y}{1+y^2},\quad d\phi=\frac{2dy}{1+y^2}.
\]
With this change of variables, we obtain
\[
\mathsf{C} f_\mu(\theta)= \frac{2x}{\pi}\,\mathrm{p.v.} \int_0^\infty \frac{y^{\mu}}{x^2-y^2}\,dy.
\]
With the further changes of variables $y=xt$ and then $s=t^2$, this integral becomes
\[
\mathsf{C} f_\mu(\theta)= \frac{2}{\pi}\, x^{\mu}\,\mathrm{p.v.} \int_0^\infty \frac{t^{\mu}}{1-t^2}\,dt=
\frac{1}{\pi}\, x^{\mu}\, \mathrm{p.v.} \int_0^\infty \frac{s^{(\mu-1)/2}}{1-s}\,ds.
\]
We use the known principal value identity
\[
\mathrm{p.v.} \int_0^\infty \frac{s^{a-1}}{1-s}\,ds=\pi \cot(\pi a),\qquad 0< \Re a <1
\]
to finally obtain
\begin{equation}	\label{eq1815fri}
\mathsf{C} f_\mu=- \tan (\pi \mu/2) f_\mu.
\end{equation}

As the image of the strip $\abs{\Re z} < \frac12$ under the mapping $z \mapsto -\tan (\pi z/2)$ is the open unit disk and the spectrum is a closed set, we have
\[
\{z \in \mathbb{C}: \abs{z} \le 1\} \subset \sigma (\mathsf{C}). 
\]
On the other hand, for every $f \in L^2_+$, we have
\[
\int_{-\pi}^\pi \abs{\mathsf{C} f(\theta)}^2\,d\theta=\int_{-\pi}^\pi \abs{\mathsf{M}_{\mathsf s} \opH f(\theta)}^2\,d\theta=
\int_{-\pi}^\pi \abs{\opH f(\theta)}^2\,d\theta \le \int_{-\pi}^\pi \abs{f(\theta)}^2 \,d\theta,
\]
where we used \eqref{eq1703fri} and Parseval's identity in the last inequality.
This implies that the operator norm $\norm{\mathsf{C}} \le 1$, and thus proves the lemma.
\end{proof}

We note that \eqref{eq1815fri} remains valid for complex $\mu$ with $\abs{ \Re \mu}<1$.
Since $\mathsf{C}=\opM_{\mathsf s}\opH$ and  $\mathsf{s}^2=1$, it follows that
\begin{equation}		\label{eq1528sat}
\opH \left( \frac{1+\sin\theta}{1-\sin\theta}\right)^{\mu/2} = -\tan\left(\frac{\pi \mu}{2}\right)\,\mathsf{s}(\theta) \left( \frac{1+\sin\theta}{1-\sin\theta}\right)^{\mu/2},\qquad \abs{\Re \mu}<1.
\end{equation}

It is also possible to show that $\sigma(\mathsf{T})=2\overline{\mathbb{D}}$, but we do not need this fact here.
\subsection{Neumann series representation of the solution}
\label{sec4.5}

Lemma \ref{lem:C_spectrum} implies that if $\abs{k}<1$, then $\opI + k \mathsf{C}$ is invertible on $L^2_+$ by the Neumann series.
\[
(\opI + k \mathsf{C})^{-1}= \sum_{n=0}^\infty (-k \mathsf{C})^n=\sum_{n=0}^\infty (-k \mathsf{M_s} \opH)^n \quad \text{on }\;L^2_+.
\]
Therefore, there is a unique $\psi \in L^2$ satisfying \eqref{eq0858mon}, and it is given by
\begin{equation}		\label{eq1714mon}
\psi=y+ i  \opH y,\quad y=\frac{1}{2i} \sum_{n=0}^\infty (-k \mathsf{M_s} \opH)^n \mathsf{\Pi}_+ g,\qquad \abs{k}<1.
\end{equation}

On the other hand, when $\abs{k}>1$, we use the identity
\[
\mathsf{C} \mathsf{C}^*= \opI\quad \text{on }\;L^2_+,\qquad \mathsf{C}^*=-\opH \mathsf{M_s} \vert_{L^2_+}.
\]
Indeed, for $f \in L^2_+$, we have
\[
\mathsf{C} \mathsf{C}^* f= \mathsf{M_s} \opH (-\opH \mathsf{M_s} f)=-\mathsf{M_s} \opH^2 \mathsf{M_s} f=-\mathsf{s} \opH^2 (\mathsf{s}f).
\]
Since $f \in L^2_+$, the function $\mathsf{s}f$ has average zero, and hence $\opH^2 (\mathsf{s}f)= -\mathsf{s}f$.
Therefore,
\[
\mathsf{C} \mathsf{C}^* f = -\mathsf{s}(-\mathsf{s}f)=\mathsf{s}^2 f=f,\qquad f \in L^2_+.
\]
However, it should be noted that $\mathsf{C}$ is not injective since
\begin{equation}		\label{eq1735mon}
\ker \mathsf{C}= \ker \opH \cap L^2_+= \operatorname{span} \{1\}.
\end{equation}
Note that for $k \neq 0$, we have
\begin{equation}		\label{eq0613mon}
\opI + k \mathsf{C}= k \left(\mathsf{C}+\frac{1}{k} \opI\right),\qquad \left(\mathsf{C}+\frac{1}{k} \opI \right) \mathsf{C}^*= \opI+ \frac{1}{k} \mathsf{C}^*.
\end{equation}
Since $\norm{\mathsf{C}^*} = 1$, if $\abs{k}>1$, then $\opI + \frac{1}{k} \mathsf{C}^*$ is invertible on $L^2_+$ by the Neumann series.
Moreover,
\[
\left(\opI+ \frac{1}{k} \mathsf{C}^*\right)^{-1}= \sum_{n=0}^\infty \left(-\frac{1}{k}\mathsf{C}^*\right)^n=\sum_{n=0}^\infty \frac{1}{k^n}(\opH \mathsf{M_s})^n,\qquad \abs{k}>1.
\]
Hence, setting $y=\mathsf{C}^*  h$, where
\[
h := \frac{1}{2ki} \left(\opI+ \frac{1}{k} \mathsf{C}^*\right)^{-1}\mathsf{\Pi}_+g,
\]
and using \eqref{eq0613mon}, we find that $y$ satisfies \eqref{eq1601mon}.
Therefore, there exists a solution $\psi \in H^2_{\rm Hardy}$ satisfying \eqref{eq0858mon}; namely,
\begin{equation}		\label{eq1606mon}
\psi=y+ i  \opH y,\qquad y=-\frac{1}{2i} \sum_{n=0}^\infty \frac{1}{k^{n+1}}(\opH \mathsf{M_s})^{n+1} \mathsf{\Pi}_+ g,\qquad \abs{k}>1.
\end{equation}
We remark that $y \in \ran \mathsf{C}^*=(\ker \mathsf{C})^\perp$, and hence, by \eqref{eq1735mon},
\begin{equation}		\label{eq1543sun}
\ip{y,1}=0.
\end{equation}
This condition serves as a normalization that selects a solution of \eqref{eq0858mon} when $\abs{k}>1$, namely the solution represented by the Neumann series \eqref{eq1606mon}.
This normalization defines a bounded linear operator $\mathsf{S}(k)$ on $L^2$ by
\begin{equation}		\label{eq1715sat}
\mathsf{S}(k)g=\tilde g= -\sum_{n=0}^\infty \frac{1}{k^{n+1}}\mathsf{H}  (\opH \mathsf{M_s})^{n+1} \mathsf{\Pi}_+ g,
\end{equation}
where we used \eqref{eq1825fri} and \eqref{eq1606mon}.

\subsection{Failure of solvability at $\boldsymbol{\abs{k} = 1}$}
\label{sec4.6}
When $\abs{k}=1$, the map $\opI +k \mathsf{C}$ fails to be onto $L^2_+$.
To see this, let $u_0=e_0=1$, and define
\[
u_n = (\mathsf{C}^*)^n u_0,\quad n \ge 0.
\]
Since $\mathsf{C}\mathsf{C}^*=\opI$ and $\ker \mathsf{C} = \operatorname{span}\{ u_0 \}$, it follows that $\{u_n\}_{n=0}^\infty$ is an orthonormal sequence in $L^2_+$.
Note that
\[
\mathsf{C} u_0=0,\qquad \mathsf{C} u_n =u_{n-1}\quad (n\ge 1).
\]
Suppose that $h$ and $y$ are functions in $L^2_+$ satisfying
\begin{equation}			\label{eq1328fri}
(\opI + k \mathsf{C})y=h.
\end{equation}
Define the coefficients
\[
y_n=\ip{y, u_n},\qquad h_n=\ip{h, u_n}.
\]
Then
\[
y_n + k y_{n+1}= h_n \qquad (n \ge 0).
\]
Now consider the function $h \in L^2_+$ defined by
\[
h=\sum_{n=0}^\infty \frac{(- \bar k)^n}{n+1}u_n=\sum_{n=0}^\infty \frac{(- \bar k)^n}{n+1} (\mathsf{C}^*)^n u_0.
\]
If a solution $y$ of the equation \eqref{eq1328fri} existed, then we would have
\[
y_n+k y_{n+1}=\frac{(- \bar k)^n}{n+1}.
\]
Iterating this identity, for $N>n$, we obtain
\[
y_n - (-k)^{N-n} y_N= \sum_{m=1}^{N-n} \frac{(-\bar k)^n}{n+m}=(-\bar k)^n \sum_{m=n+1}^N \frac{1}{m}.
\]
Since $y \in L^2_+$, we have $y_N \to 0$ as $N \to \infty$.
This is impossible, however, because the right-hand side diverges as $N \to \infty$.
Thus, \eqref{eq1328fri} has no solution for this particular choice of $h$.
In other words, $\opI + k \mathsf{C}$ is not onto $L^2_+$ when $\abs{k}=1$.

\section{Analysis of the solution operator: $\abs{k}<1$} \label{sec:small_k}
Throughout this section, we assume that $\abs{k}<1$.
In Section~\ref{sec:solution}, we saw that, when $\abs{k}<1$, there is a unique solution $\psi$ of \eqref{eq0858mon}, given by the Neumann series expansion \eqref{eq1714mon}.
By \eqref{eq1825fri}, we then obtain the following expression for $\tilde{g}=(\opI - \opR)\psi$:
\begin{equation}		\label{eq0515tue}
\tilde{g}=\opS(k) g:= \sum_{n=0}^\infty (-k)^n  \opH (\mathsf{M_s} \opH)^n \mathsf{\Pi}_+ g,\qquad \abs{k}<1.
\end{equation}
Thus, when $\abs{k}<1$, the operator $\opS(k)$ defines the map $g \to \tilde{g}$ on $L^2$.
This operator is also bounded on $L^p$, $1<p<\infty$, under a further restriction on $k$.

\subsection{$L^p$ estimates for $\opS(k)$}	\label{sec5.1}
Since $\norm{\mathsf{s}}_{L^\infty}=1$, we have
\[
\norm{\mathsf{M_s} \opH}_{L^p\to L^p} \le \norm{\opH}_{L^p\to L^p}.
\]
Also, since $\opR$ is an isometry on $L^p$ and $\mathsf{\Pi}_+=\frac{1}{2}(\opI + \opR)$, we have
\[
\norm{\mathsf{\Pi}_+}_{L^p\to L^p} = 1.
\]
Therefore, it follows from \eqref{eq0515tue} that, for $\abs{k}<1$, the operator $\opS(k)$ is bounded on $L^p$, provided that
\[
k\norm{\opH}_{L^p\to L^p}<1.
\]

It is well known, following the sharp estimates of Pichorides \cite[Theorem~3.7]{Pichorides}, that the Hilbert transform on $L^p$ satisfies
\[
\norm{\opH}_{L^p\to L^p} = \begin{cases}
\tan \frac{\pi}{2p}, & 1<p \le 2,\\[2pt]
\cot \frac{\pi}{2p}, & 2\le p <\infty.
\end{cases}
\]
Therefore, the operator $\opS(k)$ is bounded on $L^p$ whenever
\begin{equation}			\label{eq1111mon}
\abs{k}<k_p := \begin{cases}
\cot \frac{\pi}{2p}, & 1<p \le 2,\\[2pt]
\tan \frac{\pi}{2p}, & 2\le p <\infty,
\end{cases}
\end{equation}
and, in this case,
\begin{equation}			\label{eq1343mon}
\norm{\opS(k)}_{L^p\to L^p} \le C(k, p)<\infty,\qquad 1<p<\infty.
\end{equation}

Recall the definition of the Poisson non-tangential maximal function from Section~\ref{sec:weak_sol}.
It is well known that, for $1<p<\infty$,
\begin{equation}		\label{eq1455mon}
\norm{\mathcal{P}_{\ast} f}_{L^p} \le C \norm{f}_{L^p},
\end{equation}
where $C=C(p,\beta)$.
By \eqref{eq0839wed}, \eqref{eq1343mon}, and \eqref{eq1455mon}, we therefore obtain
\[
\norm{\mathcal{N}_*u}_{L^p} \le C \norm{g}_{L^p},\qquad C=C(p,\beta, k)>0.
\]
Then, as in \eqref{eq0841wed}, we obtain the non-tangential convergence
\[
\lim_{\substack{re^{i\theta} \to e^{i\phi}\\ re^{i\theta} \in \mathcal{A}_\beta(\phi)}} u(r,\theta)  = g(\phi),\qquad \text{for a.e. }\, \phi \in \partial\mathbb{D}.
\]

We conclude this subsection by addressing the case $g\in L^p(\partial\mathbb D)$ with $1<p<2$.
In this range, it is not immediate from the definition \eqref{eq1121mon} that $u\in H^1_{\mathrm{loc}}(\mathbb D)$ or that $u$ is a weak solution of $Lu=0$.
Recall from \eqref{eq2119tue} that
\[
u=U-kV+k\mathsf{s}W.
\]
Since $U$, $V$, and $W$ are harmonic in $\mathbb D$, it remains only to verify that $\mathsf{s}W\in H^1_{\mathrm{loc}}(\mathbb D)$.
Because $\tilde g$ is $\opR$-antisymmetric, the argument in the proof of Lemma~\ref{lem2117mon} shows that $W=\mathcal P_r \ast \tilde g=0$ on $\Gamma$.
The estimate \eqref{eq2022mon} therefore remains valid, and hence $u\in H^1_{\mathrm{loc}}(\mathbb D)$.
To show that $u$ is a weak solution, it suffices, as in Section~\ref{sec3.2}, to verify the transmission condition \eqref{eq1539tue} across $\Gamma$.
Since $\mathsf S(k)$ is bounded on $L^p$, the identity \eqref{eq1531wed}, initially established for $L^2$ data, extends to $L^p$ by density.
Indeed, approximating $g$ in $L^p$ by functions in $L^2 \cap L^p$, the identity follows by the $L^p$ boundedness of $\mathsf H$, $\mathsf R$, and $\mathsf M_{\mathsf s}$.

\subsection{$H^1$-regularity and positivity}	\label{sec5.2}
We first show that $g\in H^1(\partial\mathbb{D})$ implies $u \in H^1(\mathbb{D})$.
We refer to \eqref{eq1732mon} for the definition of $H^1(\partial\mathbb{D})$.

\begin{lemma}		\label{lem:H1}
Assume that $\abs{k}<1$, and let $g\in H^1(\partial\mathbb{D})$.
Then $\tilde{g}\in H^1(\partial\mathbb{D})$, and $u \in H^1(\mathbb{D})$.
Moreover,
\[
\norm{\nabla u}_{L^2(\mathbb{D})}\le C \norm{g}_{H^1(\partial\mathbb{D})}.
\]
\end{lemma}
\begin{proof}
Recall the decomposition of $u$ in \eqref{eq2119tue}:
\[
u=U-kV+k\mathsf{s}W:=\mathcal{P}_r \ast g - k \mathcal{P}_r \ast (\mathsf{s}\tilde{g}) + k \mathsf{s} \mathcal{P}_r \ast \tilde{g}.
\]
Since $U=\mathcal{P}_r \ast g(\theta)$, \eqref{eq1025tue} gives
\begin{equation}		\label{eq1935fri}
\norm{\nabla U}_{L^2(\mathbb{D})}^2  \le C \norm{g}_{H^{1/2}(\partial\mathbb{D})}^2 \le C \norm{g}_{H^1(\partial\mathbb{D})}^2,
\end{equation}
where we used the embedding $H^1(\partial\mathbb{D}) \hookrightarrow H^{1/2}(\partial\mathbb{D})$.

Next, we claim that $\tilde{g} \in H^1(\partial\mathbb{D})$.
To see this, define
\[
h_n=(\mathsf{M_s} \opH)^n \mathsf{\Pi}_+ g,\qquad n=0,1,2,\ldots.
\]
We prove by induction on $n$ that
\begin{equation}			\label{eq2131tue}
h_n \in L^2_+,\qquad \norm{h_n}_{H^1(\partial\mathbb{D})} \le \norm{g}_{H^1(\partial\mathbb{D})},\qquad n\ge 0,
\end{equation}
where we use the norm
\[
\norm{h}_{H^1(\partial\mathbb{D})}^2=\norm{h}_{L^2(\partial\mathbb{D})}^2+\norm{h'}_{L^2(\partial\mathbb{D})}^2.
\]
Since $\opR$ is an isometry on $H^1(\partial\mathbb{D})$, we have
\[
\norm{\mathsf{\Pi}_+ g}_{H^1(\partial\mathbb{D})} \le \norm{g}_{H^1(\partial\mathbb{D})}.
\]
Hence, \eqref{eq2131tue} holds for $n=0$.

By property \eqref{eq1703fri} of the Hilbert transform and Parseval's identity, we have
\begin{equation}	\label{eq2151fri}
\norm{\opH f}_{H^1(\partial \mathbb{D})} \le \norm{f}_{H^1(\partial\mathbb{D})}.
\end{equation}
Assume that \eqref{eq2131tue} holds for some $n \ge 0$.
Then, by \eqref{eq2151fri},
\[
\norm{\opH h_n}_{H^1(\partial\mathbb{D})} \le \norm{h_n}_{H^1(\partial\mathbb{D})} \le \norm{g}_{H^1(\partial\mathbb{D})}.
\]
Moreover, by \eqref{eq1609mon}, we have $\opH h_n \in L^2_-=\ran \mathsf{\Pi}_-$.
Noting that $h_{n+1}=\mathsf{M_s} \opH h_n$, we apply Lemma \ref{lem1051tue} with $f=\opH h_n$ to obtain
\[
\norm{h_{n+1}}_{H^1(\partial\mathbb{D})} \le \norm{\opH h_n}_{H^1(\partial\mathbb{D})} \le \norm{g}_{H^1(\partial\mathbb{D})}.
 \]
Together with \eqref{eq1609mon}, this implies that \eqref{eq2131tue} holds with $n$ replaced by $n+1$, completing the induction.

Using \eqref{eq0515tue}, \eqref{eq2131tue}, \eqref{eq2151fri}, and summing over $n$, we obtain
\begin{equation}	\label{eq2210fri}
\norm{\tilde{g}}_{H^1(\partial\mathbb{D})} \le \sum_{n=0}^\infty \abs{k}^n \norm{h_n}_{H^1(\partial\mathbb{D})}  \le \left( \sum_{n=0}^\infty \abs{k}^n \right)\norm{g}_{H^1(\partial\mathbb{D})} \le C \norm{g}_{H^1(\partial\mathbb{D})},\qquad \abs{k}<1,
\end{equation}
where $C=C(k)<\infty$.
Thus, $\tilde{g} \in H^1(\partial\mathbb{D})$ whenever $\abs{k}<1$.

Since $\tilde{g}=(\opI-\opR)\psi$, it follows from Lemma~\ref{lem2117mon} that
\[
\mathcal{P}_r\ast \tilde{g}=0\quad\text{on }\,\Gamma.
\]
Hence, by \eqref{eq2022mon}, \eqref{eq1025tue}, and \eqref{eq2210fri}, $W=\mathcal{P}_r \ast \tilde{g}$ satisfies
\begin{equation}		\label{eq1936fri}
\norm{\nabla(\mathsf{s}W)}_{L^2(\mathbb{D})}^2 =\norm{\nabla W}_{L^2(\mathbb{D})}^2  \le C \norm{\tilde{g}}_{H^{1/2}(\partial\mathbb{D})}^2 \le C \norm{g}_{H^1(\partial\mathbb{D})}^2.
\end{equation}
Also, by \eqref{eq1025tue} and \eqref{eq2210fri}, $V=\mathcal{P}_r \ast (\mathsf{s} \tilde g)$ satisfies
\begin{equation}		\label{eq1939fri}
\norm{\nabla V}_{L^2(\mathbb{D})}^2  \le C \norm{\mathsf{s}\tilde{g}}_{H^{1/2}(\partial\mathbb{D})}^2 \le C \norm{g}_{H^1(\partial\mathbb{D})}^2.
\end{equation}
Combining \eqref{eq1935fri}, \eqref{eq1936fri}, and \eqref{eq1939fri}, we complete the proof.
\end{proof}

We next prove positivity.
For sufficiently regular nonnegative boundary data, the corresponding solution belongs to $H^1(\mathbb D)$, and hence the weak maximum principle gives $u \ge 0$ in $\mathbb D$.
By the standard approximation procedure for the continuous Dirichlet problem, this yields, for every $g\in C(\partial\mathbb D)$, a unique solution $u\in C(\overline{\mathbb D})$, and the resulting solution operator is positivity preserving; see \cite[Chapter~8]{GT}.
Consequently, for each $re^{i\theta} \in\mathbb D$, the map
\[
g\longmapsto u(r,\theta)
\]
is represented by a positive $L$-harmonic measure on $\partial\mathbb D$.

Finally, under the assumptions of Section~\ref{sec5.1}, the $L^p$ solution agrees with this representation.
Therefore, if $g\in L^p(\partial\mathbb D)$ and $g \ge 0$ a.e. on $\partial\mathbb D$, then $u \ge 0$ in $\mathbb D$.

\subsection{Summary}
We summarize the results obtained in the regime $\abs{k}<1$ in the following theorem.
We note that the range $\abs{k}<k_p$ in part~(a) is not optimal when $p>2$; this issue will be revisited in Section~\ref{sec:RH}.

\begin{theorem}	\label{thm:small_k}
Let $\abs{k}<1$, and define
\[
u(r,\,\cdot\,)=\mathsf{P_r} g-k[\mathsf{P_r}, \mathsf{M_s}]\opS(k)g,
\]
where $\opS(k)$ is given by \eqref{eq0515tue}.
Then the following properties hold:
\begin{enumerate}
\item \emph{($L^p$ boundedness and non-tangential convergence).}
Let $1<p<\infty$, and suppose that $\abs{k}<k_p$, where $k_p$ is defined in \eqref{eq1111mon}.
Then $\opS(k)$ is bounded on $L^p(\partial\mathbb D)$.
For every $g\in L^p(\partial\mathbb D)$, the corresponding function $u$ belongs to $H^1_{\mathrm{loc}}(\mathbb D)$, is a weak solution of \eqref{eq:u}, and converges non-tangentially to $g$ almost everywhere on $\partial\mathbb{D}$.
Moreover,
\[
\norm{\mathcal{N}_*u}_{L^p(\partial\mathbb{D})} \le C \norm{g}_{L^p(\partial\mathbb{D})},
\]
where $C=C(p,k,\beta)>0$ and $\beta$ is the aperture constant in \eqref{eq0856wed}.

\item \emph{($H^1$-regularity and positivity).}
If $g\in H^1(\partial\mathbb D)$, then $u\in H^1(\mathbb D)$ and is the unique weak solution in $H^1(\mathbb D)$ of the Dirichlet problem \eqref{eq:dp_in_disk}.
Moreover, under the assumptions of part~{\rm (a)}, the solution operator is positivity preserving: if $g \in L^p(\partial\mathbb D)$ and $g\ge 0$ almost everywhere on $\partial\mathbb D$, then $u\ge 0$ in $\mathbb D$.
\end{enumerate}
\end{theorem}

\section{Riemann--Hilbert problem}			\label{sec:RH}
We saw that finding $\psi \in H^2_{\rm Hardy}$ satisfying the equation \eqref{eq1555mon} is equivalent to finding $y \in L^2_+$ satisfying the equation \eqref{eq1601mon}.

\subsection{Explicit solution formula}
We are interested in the case when $g$ is real-valued.
Multiplying \eqref{eq1601mon} by $i$, we obtain
\begin{equation}	\label{eq1835mon}
q+k\mathsf{s} \opH q= \tfrac{1}{2} \mathsf{\Pi}_+g,\qquad q=iy,
\end{equation}
where $q$ and $g$ are real-valued functions on the unit circle $\partial \mathbb{D}$.

We next recast the boundary equation as a Riemann--Hilbert problem, following a classical reduction in the theory of singular integral equations; see, for example, \cite{Gakhov,Muskhelishvili}.
By Lemma \ref{lem:B_isometry}, we may associate to $q$ a holomorphic function $Y: \mathbb{D} \to \mathbb{C}$ whose boundary trace satisfies
\[
Y_+= q+i \opH q \in H^2_{\rm Hardy}.
\]
Equation \eqref{eq1835mon} then takes the form
\begin{equation}	\label{eq1853mon}
\Re [(1-ik\mathsf{s})Y_+]=\tfrac{1}{2} \mathsf{\Pi}_+g.
\end{equation}
Since $Y$ is holomorphic in $\mathbb D$ with boundary trace $q+i\opH q$, it has the expansion
\[
Y(z)=  \hat q(0)+2\sum_{n = 1}^\infty \hat q(n) z^n.
\]
In particular, $Y$ must satisfy the normalization condition
\begin{equation}	\label{eq1539wed}
Y(0)=\hat q(0)=\ip{q,1} \in \mathbb{R}.
\end{equation}

We factor $Y$ as $Y=XZ$, where
\[
X(z):=\left( \frac{1-iz}{1+iz} \right)^\mu,\qquad X(0)=1,\qquad \mu \in \mathbb{R}.
\]
Here, the power is defined using the holomorphic branch of the logarithm on the right half-plane. 
Indeed, the M\"obius transformation
\[
z\mapsto \frac{1-iz}{1+iz}
\]
maps $\mathbb{D}$ conformally onto $\{z \in \mathbb{C}: \Re z>0\}$.
Thus $X$ is holomorphic and nonvanishing in $\mathbb{D}$.
Moreover, for $z=e^{i\theta}$,
\[
\frac{1-ie^{i\theta}}{1+ie^{i\theta}}= -i \,\frac{\cos\theta}{1-\sin\theta},
\]
and hence its boundary trace satisfies
\[
\arg X_+(\theta) =-\frac{\mu\pi}{2}\mathsf{s}(\theta),\qquad \theta \neq  \pm \frac{\pi}{2}.
\]

Since
\[
\arg(1-ik\mathsf{s})=-\mathsf{s} \alpha,\qquad \alpha:=\arctan k \in (-\pi/2, \pi/2),
\]
we have
\[
(1-ik\mathsf{s})X_+= \sqrt{1+k^2}\, \abs{X_+}\, e^{-i\mathsf{s}(\alpha+\mu\pi/2)}.
\]
We choose
\begin{equation}		\label{eq1854mon}
\mu=-\frac{2\alpha}{\pi} + N,\qquad N \in \mathbb{Z},
\end{equation}
so that \eqref{eq1853mon} becomes
\begin{equation}	\label{eq1855mon}
\Re [e^{-\mathsf{s} N \pi i /2} Z_+]=f:=\frac{\mathsf{\Pi}_+g}{2\sqrt{1+k^2}\, \abs{X_+}},
\end{equation}
where $Z$ is holomorphic in $\mathbb{D}$ and $Z_+$ denotes its boundary trace on $\partial \mathbb{D}$.
Note that $e^{-\mathsf{s} N \pi i /2} = \pm 1$ or $e^{-\mathsf{s} N \pi i /2} = \pm i$, depending on whether $N$ is even or odd.

To find a function $Z$ satisfying \eqref{eq1855mon}, introduce the Schwarz integral operator
\begin{equation}	\label{eq1859mon}
\mathcal{S} f(z)= \frac{1}{2\pi} \int_{-\pi}^\pi f(e^{i\theta})\, \frac{e^{i\theta}+z}{e^{i\theta}-z} \, d\theta,\qquad z \in \mathbb{D}.
\end{equation}
We require $f \in L^{1+}(\partial \mathbb{D})$, that is, $f \in L^{1+\epsilon}(\partial \mathbb{D})$ for some $\epsilon>0$, so that
\begin{equation}	\label{eq1041wed}
(\mathcal{S}f)_+=f+ i\opH f.
\end{equation}
Since
\[
\abs{X_+(e^{i\theta})}= \Abs{\frac{\cos \theta}{1-\sin \theta}}^\mu=
\left(\frac{1+\sin \theta}{1-\sin\theta} \right)^{\mu/2}
\]
and $g \in L^2(\partial \mathbb{D})$, we require $\mu \in (-\frac12, \frac12)$ in order to apply H\"older's inequality and obtain
\[
1/\abs{X_+} \in L^{2+}(\partial \mathbb{D}).
\]
We will also need the following explicit formula for $X_+$:
\begin{equation}	\label{eq1328wed}
X_+(e^{i\theta})= \abs{X_+(e^{i\theta})}\, e^{-i\pi \mu \mathsf{s}(\theta)/2}= \left(\frac{1+\sin \theta}{1-\sin\theta} \right)^{\mu/2} e^{-i\pi \mu \mathsf{s}(\theta)/2}.
\end{equation}

We consider separately the cases $\abs{k}<1$ and $\abs{k}>1$.

\subsubsection*{Case 1: $\abs{k}<1$}
We take $N=0$, so that $\mu=-2\alpha/\pi \in (-\frac12, \frac12)$. 
Then \eqref{eq1855mon} becomes
\[
\Re Z_+=f,\qquad f=\frac{\mathsf{\Pi}_+ g}{2\sqrt{1+k^2}\, \abs{X_+}}.
\]
Hence $Z-\mathcal{S}f$ is holomorphic in $\mathbb{D}$ and satisfies
\[
\Re (Z-\mathcal{S}f)_+ =0\quad \text{a.e.  on }\;\partial\mathbb{D}.
\]
It follows that
\begin{equation}		\label{eq0807thu}
Z-\mathcal{S}f= i C,\quad  C \in \mathbb{R}.
\end{equation}
Therefore,
\[
Y=XZ=X \mathcal{S}f + i C X.
\]
Since $X(0)=1$, we have
\[
Y(0)= \mathcal{S}f(0) + iC.
\]
It follows from \eqref{eq1859mon} that $\mathcal{S}f(0)=\ip{f,1} \in \mathbb{R}$.
Thus, by the condition  \eqref{eq1539wed}, we obtain $C=0$.
Therefore, $Z=\mathcal{S}f$, and hence, by \eqref{eq1041wed},
\[
Y_+=X_+ f + i X_+  \opH f,\qquad f=\frac{\mathsf{\Pi}_+ g}{2\sqrt{1+k^2}\, \abs{X_+}}.
\]
Since $y=-i \Re Y_+$, using \eqref{eq1328wed} with $\mu=-2\alpha/\pi$, we obtain
\begin{multline}	\label{eq1910mon}
y=-\frac{i}{2(1+k^2)} \left\{ \mathsf{\Pi}_+ g -k \mathsf{s} w\, \opH \left(\frac{\mathsf{\Pi}_+ g}{w}\right) \right\},\\
w(\theta)= \left(\frac{1+\sin \theta}{1-\sin\theta} \right)^{\mu/2},\qquad \mu=-\frac{2}{\pi}\arctan k.
\end{multline}
Since $w^2$ is an $A_2$ weight for $\abs{\mu}<1/2$, it follows from \eqref{eq1910mon} that $y \in L^2_+$ whenever $g \in L^2$; see \cite{HMW1973}.
Therefore, $\psi=y+i \opH y \in H^2_{\rm Hardy}$ whenever $g \in L^2$.
By uniqueness, this solution must agree with the Neumann series solution \eqref{eq1714mon}. 

\subsubsection*{Case 2: $\abs{k}>1$}
We take $N=\sgn k \in \{1, -1\}$, so that
\[
\mu=-\tfrac{2}{\pi} \arctan k + \sgn k \in (-\tfrac12,  \tfrac12).
\]
Since $\mathsf{s}$ is real-valued and $\mathsf{s}^2=1$, equation \eqref{eq1855mon} becomes
\[
\Im Z_+=-\Re [i  Z_+]= (\sgn k)\, \mathsf{s}f,\qquad f =\frac{\mathsf{\Pi}_+ g}{2\sqrt{1+k^2}\, \abs{X_+}}.
\]
Using the identity $\Im[i \mathcal{S}(\mathsf{s}f)]_+=\mathsf{s}f$, we obtain, as in \eqref{eq0807thu},
\[
Z-i (\sgn k) \, \mathcal{S}[\mathsf{s}f]=C,\qquad C \in \mathbb{R}.
\]
Therefore,
\[
Y=XZ=CX +i(\sgn k)\, X \mathcal{S}[\mathsf{s}f], \qquad C \in \mathbb{R}.
\]
Using $X(0)=1$ and $\mathcal{S}[\mathsf{s}f](0)= \ip{\mathsf{s}f, 1}=0$, we obtain $Y(0)= C \in \mathbb{R}$.
Hence the condition \eqref{eq1539wed} is automatically satisfied, and
\begin{equation}	\label{eq0942thu}
C=\ip{q,1}=i\ip{y,1}.
\end{equation}
Therefore, using \eqref{eq1328wed} with $\mu=-2\alpha/\pi+ \sgn k$, as in \eqref{eq1910mon}, we obtain
\begin{equation}			\label{eq1437mon}
y= -i \left\{C \frac{\abs{k} w}{\sqrt{k^2+1}}  + \frac{\mathsf{\Pi}_+ g}{2(k^2+1)} - \frac{kw }{2(k^2+1)} \opH \left(\frac{\mathsf{s}\mathsf{\Pi}_+ g}{w}\right) \right\},
\end{equation}
where
\[
w(\theta)= \left(\frac{1+\sin \theta}{1-\sin\theta} \right)^{\mu/2},\qquad \mu=\sgn k -\frac{2}{\pi}\arctan k.
\]

Since $w \in L^2_+$ and $w^2$ is an $A_2$ weight, the same argument used in the case $\abs{k}<1$ shows from \eqref{eq1437mon} that $y \in L^2_+$ whenever $g \in L^2$.
Therefore, $\psi=y+i \opH y$ is a solution of \eqref{eq0858mon}.
In contrast to the case when $\abs{k}<1$, the constant $C$ remains a free real parameter.
This is consistent with the failure of injectivity of the operator $\opI + k \mathsf{C}$ when $\abs{k}>1$.
Note that $w$ is precisely the eigenfunction appearing in the proof of Lemma~\ref{lem:C_spectrum}.
In particular, if $y$ is the solution given by \eqref{eq1606mon}, then \eqref{eq0942thu} and \eqref{eq1543sun} imply that $C=0$.
Hence
\begin{multline}			\label{eq1344fri}
y= -\frac{i}{2(k^2+1)} \left\{ \mathsf{\Pi}_+ g - k w\, \opH \left(\frac{\mathsf{s}\mathsf{\Pi}_+ g}{w}\right) \right\},\\
w(\theta)= \left(\frac{1+\sin \theta}{1-\sin\theta} \right)^{\mu/2},\qquad \mu=\sgn k -\frac{2}{\pi}\arctan k.
\end{multline}
Therefore, for this canonical choice of $y$, the function $\psi= y + i\opH y$ coincides with the Neumann series solution \eqref{eq1606mon}.

\subsection{Poisson kernel for the Dirichlet problem}
\label{sec6.2}
We now analyze the results obtained in the previous subsection. We write
\[
L^2(\partial\mathbb D;\mathbb R)
:=\left\{g\in L^2(\partial\mathbb D): g \text{ is real-valued a.e.}\right\}.
\]
Let $g\in L^2(\partial\mathbb D;\mathbb R)$, and define $y$ by \eqref{eq1910mon} when $\abs{k}<1$ and by \eqref{eq1344fri} when $\abs{k}>1$.
Then $\psi=y+i\opH y$ satisfies \eqref{eq0858mon} and admits the series representations \eqref{eq1714mon} and \eqref{eq1606mon} in the regimes $\abs{k}<1$ and $\abs{k}>1$, respectively.
Together with \eqref{eq1825fri}, this shows that $\tilde{g}=(\opI-\opR)\psi$ is real-valued and satisfies
\begin{equation}		\label{eq0228wed}
\norm{\tilde{g}}_{L^2(\partial\mathbb{D})} \le C \norm{g}_{L^2(\partial\mathbb{D})},
\end{equation}
where $C$ depends only on $k$.
Indeed,
\[
\tilde{g}=\mathsf{S}(k)g,
\]
where $\mathsf{S}(k)$ is defined by \eqref{eq0515tue} when $\abs{k}<1$ and by \eqref{eq1715sat} when $\abs{k}>1$.
Hence, the function $u$ defined by \eqref{eq1121mon} is real-valued and satisfies conclusions (a)--(d) of Theorem~\ref{thm1}, together with the estimate \eqref{eq0228wed}.

For each fixed $re^{i\theta}\in\mathbb D$, the map $g \mapsto u(r,\theta)$ defines a bounded linear functional on the real Hilbert space $L^2(\partial\mathbb D; \mathbb R)$.
Therefore, by the Riesz representation theorem, there exists a unique function
\[
\mathcal K_k(r, \theta,\,\cdot\,)\in L^2(\partial\mathbb D; \mathbb R),
\]
which we call the Poisson kernel, such that
\begin{equation}\label{poisson_kernel}
u(r,\theta) = \frac{1}{2\pi}\int_{-\pi}^{\pi} \mathcal K_k(r, \theta,\phi)g(\phi)\,d\phi.
\end{equation}
When the dependence on $r$, $\theta$, and $\phi$ is not important, we shall write simply $\mathcal K_k$.

For $\phi\neq\pm\pi/2$, an explicit formula for $\mathcal K_k(r, \theta,\phi)$ may be derived formally by taking $g=\delta(\,\cdot-\phi)$.
Since the formula will not be used below, we omit it.

We note that $\mathcal{K}_k$ depends on the index $N$ appearing in \eqref{eq1854mon}.
To indicate the chosen branch, we write
\[
\mathcal{K}_k(r, \theta,\,\cdot\,)= \mathcal{K}_k^{(N)}(r, \theta,\,\cdot\,).
\]
For $L^2$ boundary data, the natural choice is
\[
N=0\;\text{ when }\;\abs{k}<1, \qquad N=1\;\text{ when }\; k>1, \qquad N=-1\; \text{ when }\; k<-1.
\]
The admissible choice of $N$, however, is not confined to these regimes.
Other branches may also be used under suitable $L^p$-integrability assumptions on the boundary data.

We first consider the branch $\mathcal{K}_k^{(0)}$, without imposing the restriction $\abs{k}<1$.
The following lemma, together with \eqref{eq1825fri}, shows that if $y$ is defined by \eqref{eq1910mon}, then
\[
\tilde{g} = 2i \opH y \in H^{1/2+}(\partial \mathbb{D})
\]
whenever $g \in H^{1/2+}(\partial \mathbb{D})$, for every $k\in\mathbb R$.

\begin{lemma}	\label{lem_sobolev}
Let $-1<\mu<1$, and define
\begin{equation}		\label{eq1752sat}
\mathcal{T}_\mu g =\mathsf{s} w \,\opH\left(\frac{\mathsf{\Pi}_+g}{w}\right),
\qquad
w(\theta)=\left(\frac{1+\sin\theta}{1-\sin\theta}\right)^{\mu/2}.
\end{equation}
If $g \in H^{1/2+\epsilon}(\partial \mathbb{D})$ for some $\epsilon>0$, then
\[
\mathcal{T}_\mu g  \in H^{1/2+\delta}(\partial \mathbb{D})
\]
for every $\delta$ satisfying
\[
0<\delta\le\epsilon \qquad\text{and}\qquad \delta<1-\abs{\mu}.
\]
\end{lemma}

The proof proceeds by localizing near the singular points $\theta=\pm\pi/2$.
Away from these points, the claim follows from the standard Sobolev boundedness of $\opH$ and multiplication by smooth functions.
Near each singularity, the operator reduces to a localized power-weighted Hilbert transform of the form
\[
f\mapsto \sgn (t)\, \abs{t}^{-\nu}\, \mathcal{H} (\abs{\,\cdot\,}^\nu f)(t),\qquad \mathcal{H}f(t)=\frac{1}{\pi}\, \mathrm{p.v.}\! \int_{-\infty}^\infty \frac{f(s)}{t-s}\,ds,
\]
on even functions, with $\nu=\mu$ near $\theta=\pi/2$ and $\nu=-\mu$ near $\theta=-\pi/2$.
A localization argument and standard estimates for power-weighted Hilbert transforms yield the stated range of $\delta$.
The corresponding model also appears in the proof of Lemma~\ref{lem:C_spectrum}, where the same principal value computation gives
\[
\mathcal T_\mu 1=\tan(\pi\mu/2).
\]
We omit the technical details.
In the special case $\abs{\mu}<\frac12$ and $\epsilon=\delta=\frac12$, an elementary proof is given in Lemma~\ref{lem:H1}.

\smallskip
Theorem~\ref{thm1}(e) therefore shows that, for every $k\in\mathbb R$, if the boundary datum $g$ is real-valued and belongs to $H^{1/2+}(\partial\mathbb D)$, then the solution
\[
u=\mathsf{P_r} g-k[\mathsf{P_r},\mathsf{M_s}] \tilde{g}, \qquad \tilde{g}=2i\opH y,
\]
with $y$ given by \eqref{eq1910mon}, satisfies the weak maximum principle.
Equivalently, the branch $\mathcal{K}_k^{(0)}$ defines a nonnegative Poisson kernel even outside the regime in which it is the natural $L^2$ branch.

\subsection{$L^p$ Dirichlet problem}
As noted in the preceding subsection, the branch $\mathcal{K}_k^{(0)}$ is not intrinsically confined to the regime $\abs{k}<1$.
That restriction was imposed in the $L^2$ construction to ensure that $y$ defined by \eqref{eq1910mon} belongs to $L^2(\partial\mathbb D)$ and that
\[
\tilde g =2i \opH y \in L^2(\partial\mathbb D)
\]
whenever $g\in L^2(\partial\mathbb D)$.
As explained in Section~\ref{sec5.1}, however, one may instead take $g \in L^p(\partial\mathbb D)$ for $1<p<\infty$, provided that $k$ lies in the corresponding admissible range.
The key point is the $L^p$-boundedness of the operator $\opS(k)$, which maps $g$ to $\tilde g$.
Since $\tilde g=2i\opH y$ and $\opH$ is bounded on $L^p$ for every $1<p<\infty$, it suffices to establish the $L^p$-boundedness of the map $g \mapsto y$.
For the branch $N=0$, this follows from the mapping property of $\mathcal T_\mu$ defined in \eqref{eq1752sat} with $\mu=-\frac{2}{\pi} \arctan k$:
\begin{equation}	\label{eq1502wed}
\mathcal T_\mu: L^p(\partial \mathbb D) \to L^p(\partial \mathbb D) \iff  \frac{1}{1-\abs{\mu}}<p<\infty.
\end{equation}
Consequently, if $g\in L^p(\partial\mathbb D)$ in this range, then
\[
y \in L^p(\partial\mathbb D) \qquad\text{and}\qquad \tilde g=2i\opH y\in L^p(\partial\mathbb D).
\]
It is worth noting that the admissible range of $p$ for $\mathcal T_\mu$ is larger than that for the weighted Hilbert transform
\[
f\mapsto w \opH(f/w),
\]
whose $L^p$-boundedness is characterized by the condition $w^p\in A_p$, or equivalently,
\[
\abs{\mu} < \min\left\{\frac{1}{p}, 1-\frac{1}{p}\right\}.
\]
In particular, we used this criterion to conclude that $w^2 \in A_2$ whenever $\abs{\mu}<1/2$.
The larger range for $\mathcal T_\mu$ results from the restriction to $f=\mathsf{\Pi}_+ g$ together with the presence of the factor $\mathsf{s}$.
Indeed, after exploiting the $\opR$-symmetry and making the change of variables $t=\sin\theta$, as in \eqref{eq0854thu}, the problem reduces to the $L^p$-boundedness of a finite Hilbert transform with a Jacobi-type weight:
\[
F(t) \mapsto W(t) \,\frac{1}{\pi}\,\mathrm{p.v.}\! \int_{-1}^{1} \frac{F(s)}{(t-s)W(s)} \,ds,
\]
where
\[
W(t)=(1+t)^a(1-t)^b,\qquad a=\frac12\left(1+\mu-\frac1p\right),\qquad b=\frac12\left(1-\mu-\frac1p\right).
\]
The condition $W^p \in A_p((-1,1))$ is equivalent to $\abs{\mu} < 1- 1/p$; see \cite{HMW1973}.
Moreover, under the correspondence
\[
F(t)= f(\arcsin t) (1-t^2)^{-1/2p},
\]
we have $F \in L^p((-1,1))$ if and only if $f \in L^p((-\pi/2,\pi/2))$.
The assertion \eqref{eq1502wed} follows from these observations.

Consequently, if $g \in L^p(\partial \mathbb D)$ for some $p \in (p_k,\infty)$, where
\begin{equation}	\label{erq1740wed}
p_k:=\frac{\pi}{\pi-2 \arctan \abs{k}},
\end{equation}
then \eqref{eq0803wed} and \eqref{eq1455mon} imply that the function
\begin{equation}		\label{eq1553sun}
u(r,\cdot) = \mathsf{P_r} g -k[\mathsf{P_r}, \mathsf{M_s}] \tilde{g}
\end{equation}
satisfies
\[
\norm{\mathcal{N}_*u}_{L^p} \le C \norm{g}_{L^p}, \qquad  p_k <p<\infty,
\]
where $C=C(k,p,\beta)$ and $\beta$ is the aperture constant in \eqref{eq0856wed}.

The endpoint case $p=\infty$ follows by combining the weak maximum principle established in Section~\ref{sec6.2} with the positivity argument in Section~\ref{sec5.2}.
Together with a standard approximation argument, the maximum principle yields unique solvability of the Dirichlet problem for continuous boundary data and thereby defines the associated $L$-harmonic measure.
Accordingly, the kernel $\mathcal K_k^{(0)}$ is nonnegative and represents the density of this measure.

For fixed $p$, the condition $p>p_k$ is equivalent to
\[
\abs{k} <\cot(\pi/2p).
\]
This agrees with the range in Theorem~\ref{thm:small_k} when $1<p\le2$, but is strictly larger when $p>2$.

Similarly, the branch $\mathcal K^{(1)}_k$ is not confined to the regime $k>1$: under a suitable $L^p$-integrability assumption on $g$, it may be used for every $k>0$.
The analogous statement holds for the branch $\mathcal K_k^{(-1)}$ when $k<0$.
More precisely, for $-1<\mu<1$, $\mu\ne 0$, define
\[
\tilde{\mathcal T}_\mu g :=w \,\opH\left(\frac{\mathsf{s} \mathsf{\Pi}_+g}{w}\right),\qquad
w(\theta)=\left(\frac{1+\sin\theta}{1-\sin\theta}\right)^{\mu/2}.
\]
Then $\tilde{\mathcal T}_\mu$ is bounded on $L^p(\partial\mathbb D)$ if and only if $1<p<1/\abs{\mu}$.
This follows once again by exploiting the $\opR$-symmetry and applying the change of variables $t=\sin \theta$.
In this case, however, instead of \eqref{eq1225thu}, we use the identity
\[
\cot\left(\frac{\theta-\phi}{2}\right) + \tan\left(\frac{\theta+\phi}{2}\right) = \frac{2\cos\phi}{\sin\theta-\sin\phi},
\]
which leads to the Jacobi-type weight
\[
W(t)=(1+t)^a(1-t)^b,\qquad a=\frac12\left(\mu-\frac1p\right),\qquad b=\frac12\left(-\mu-\frac1p\right).
\]
For $1<p<\infty$, the condition $W^p \in A_p((-1,1))$ is equivalent to $\abs{\mu} < 1/p$.

Consequently, if $g \in L^p(\partial \mathbb D)$ for some $p \in (1,p_k)$,
then $y$, defined in \eqref{eq1344fri} with $\mu=\sgn k -\frac{2}{\pi} \arctan k$, belongs to $L^p(\partial \mathbb D)$.
Hence, if the branch $\mathcal K_k^{(1)}$ is used for $k>0$ or the branch $\mathcal K_k^{(-1)}$ for $k<0$, then the corresponding function $u$ in \eqref{eq1553sun} satisfies
\[
\norm{\mathcal{N}_*u}_{L^p} \le C \norm{g}_{L^p}, \qquad 1<p<p_k,
\]
where $C=C(k, p, \beta)$.
Together with \eqref{erq1740wed}, this identifies $\abs{k}=\cot(\pi/2p)$ as the branching threshold in the $L^p$ theory.
In particular, for $p=2$, this recovers the $L^2$ threshold $\abs{k}=1$.

Finally, we verify that the function $u$ defined by \eqref{eq1553sun} is a weak solution of $Lu=0$ in $\mathbb D$.
By the same argument as in Section~\ref{sec5.1}, we have $u\in H^1_{\mathrm{loc}}(\mathbb D)$, so it remains only to verify the transmission condition \eqref{eq1539tue}.
We do so directly using the explicit formulas for $y$.

We shall use the following form of Carleman's identity.
Suppose that
\[
f \in L^q,\quad h \in L^r,\qquad q,r>1,\qquad 1/q+1/r < 1.
\]
Then
\begin{equation}\label{eq1729fri} \mathsf H(f\mathsf Hh)+\mathsf H(h\mathsf Hf) =(\mathsf Hf)(\mathsf Hh)-fh+\hat f(0)\hat h(0).
\end{equation}

Recall that $\tilde g=2i\mathsf H y$ is $\mathsf R$-antisymmetric.
Using $\mathsf H\mathsf R=-\mathsf R\mathsf H$ and \eqref{eq1112sun}, we see that \eqref{eq1531wed} is equivalent to
\begin{equation}\label{eq1119sun}
\mathsf H( 2iy+2ik\mathsf s\mathsf H y-\mathsf{\Pi}_+g)=0.
\end{equation}
We first assume that $g$ is smooth.
For the branch $N=0$, we use \eqref{eq1910mon} together with
\[
\mathsf H w=k\mathsf s w,
\]
whereas for the branch $N=\sgn k$ with $k\neq 0$, we use \eqref{eq1344fri} together with
\[
\mathsf H w=-\frac1k\mathsf s w.
\]
Applying Carleman's identity \eqref{eq1729fri} in each case yields
\begin{equation}	\label{eq1327sun}
2iy+2ik\mathsf s\mathsf H y-\mathsf{\Pi}_+g=0,
\end{equation}
and hence \eqref{eq1119sun}.
The constant term in Carleman's identity vanishes because one of the relevant factors is $\mathsf R$-antisymmetric and therefore has zero average.
We next extend \eqref{eq1327sun} to every $g\in L^p(\partial\mathbb D)$ in the admissible range of $p$.
Indeed, smooth functions are dense in $L^p$, while the map $g\mapsto y$ and the operators $\opH$ and $\mathsf{\Pi}_+$ are bounded on $L^p$.
Thus, by approximation, \eqref{eq1327sun} continues to hold for general $L^p$ data.
It follows that the condition \eqref{eq1539tue} is satisfied.
Therefore $u$ is a weak solution of $Lu=0$ in $\mathbb D$.

\subsection{Sign-changing Poisson kernel}
As noted in the final paragraph of Section~\ref{sec6.2}, the maximum principle holds when the branch $\mathcal{K}^{(0)}_k$ is used in \eqref{poisson_kernel} and 
\[
g \in H^{1/2+}(\partial\mathbb{D}).
\]
This conclusion need not hold for the other branches.
We first verify directly that the constant boundary datum $g\equiv1$ gives the constant solution $u \equiv 1$ for the branch $\mathcal K^{(0)}_k$, and then compare this with the solutions arising from $\mathcal{K}^{(1)}_k$ and $\mathcal{K}^{(-1)}_k$.

\smallskip
Since $\mathsf{\Pi}_+g\equiv1$, formula \eqref{eq1910mon} for $y$ is valid for all $k\in\mathbb{R}$.
Thus
\[
y=-i\left\{\frac{1}{2(1+k^2)} - \frac{k \mathsf{s}}{2(1+k^2)}\, w\, \opH \left(\frac{1}{w}\right)\ \right\}.
\]
Using \eqref{eq1528sat}, we have
\[
\opH \left(\frac{1}{w}\right)= -\frac{k\mathsf{s}}{w},\qquad w(\theta)= \left(\frac{1+\sin \theta}{1-\sin\theta} \right)^{\mu/2},\qquad \mu=-\frac{2}{\pi}\arctan k.
\]
It follows that $y=-i/2$.
Consequently, \eqref{eq1825fri} gives $\tilde{g}=2i \opH y=0$, and hence \eqref{eq1121mon} yields $u\equiv 1$.
Thus, for the branch $\mathcal{K}^{(0)}_k$, the constant boundary datum $g \equiv 1$ produces the constant solution $u\equiv 1$, in accordance with the maximum principle.

By contrast, when $k \neq 0$, formula \eqref{eq1344fri} with $g \equiv 1$ gives
\begin{equation}		\label{eq2223mon}
y= -i \left\{\frac{1}{2(k^2+1)} - \frac{k}{2(k^2+1)} w\, \opH \left(\frac{\mathsf{s}}{w}\right) \right\}.
\end{equation}
Using \eqref{eq1528sat} together with the identity $\opH^2 f = -f +\hat{f}(0)$, we obtain
\[
\opH \left(\frac{1}{w}\right)= \frac{\mathsf{s}}{k w},\qquad 
\opH \left(\frac{\mathsf{s}}{w}\right)= \sgn k \sqrt{k^2+1}-\frac{k}{w}.
\]
Substituting the second identity into \eqref{eq2223mon}, we find
\[
y=-\frac{i}{2}\,\left(1- \frac{\abs{k} w}{\sqrt{k^2+1}}\right),\qquad w(\theta)=\left(\frac{1+\sin \theta}{1-\sin\theta} \right)^{\mu/2},\qquad \mu=\sgn k - \frac{2}{\pi}\arctan k.
\]
It then follows from \eqref{eq1528sat} that
\[
\tilde{g}=2i \opH y=- \frac{\abs{k}}{\sqrt{k^2+1}}\,\opH w= \frac{\sgn k }{\sqrt{k^2+1}}\,\mathsf{s} w.
\]
Therefore, by \eqref{eq2119tue} and \eqref{eq:uvw},
\begin{multline}		\label{eq2341mon}
u =1 - \frac{\abs{k}}{\sqrt{k^2+1}}\, \{ \mathcal{P}_r \ast  w -  \mathsf{s} \mathcal{P}_r \ast(\mathsf{s}w)\},\\
w(\theta)=\left(\frac{1+\sin \theta}{1-\sin\theta} \right)^{\mu/2},\qquad \mu:= \sgn k - \frac{2}{\pi}\arctan k.
\end{multline}

For $k\neq0$, the function $u$ in \eqref{eq2341mon} is a genuine solution that converges non-tangentially to $1$ a.e. on $\partial\mathbb{D}$, yet it is not identically equal to $1$.
Thus, both the constant solution $u \equiv 1$ and the solution \eqref{eq2341mon} are admissible.
More generally, a branch $\mathcal{K}^{(N)}_k$ may be used whenever
\[
-\frac{2}{\pi}\arctan k+N\in(-1,1),
\]
provided that the boundary datum $g$ has sufficient integrability.
The preceding examples make this phenomenon explicit and show that uniqueness fails in this broader class of solutions.

Consequently, the Poisson kernel is not unique in this enlarged class.
This non-uniqueness is intrinsic to the choice of branch in the Riemann--Hilbert factorization and is not merely a consequence of crossing the $L^2$ threshold $\abs{k}=1$.

Moreover, formula \eqref{eq2341mon} shows directly that the solution $u$ changes sign.
Indeed, when $k>0$, we have $0<\mu<1$, and
\[
u(r, \pi/2)=1-\left( \frac{1+r}{1-r}\right)^\mu<0,\qquad u(r, -\pi/2)=1-\left( \frac{1-r}{1+r}\right)^\mu>0,\qquad 0<r<1.
\]
When $k<0$, we have $-1<\mu<0$, and the two inequalities are reversed.
Thus, in either case, the solution corresponding to the positive boundary datum $g\equiv1$ takes both positive and negative values.
It follows from the Poisson representation that the kernels $\mathcal K_k^{(1)}$ and $\mathcal K_k^{(-1)}$ change sign.

The preceding argument establishes the sign change of the kernel associated with the canonical normalization, corresponding to  $C=0$ in the Riemann--Hilbert factorization.
We now show that the sign change persists under every bounded linear selection of the free parameter.

For an arbitrary scalar $C$, the corresponding function $\tilde g$ is given by
\[
\tilde g = 2i\opH y + 2C\frac{\sgn k}{\sqrt{k^2+1}}\mathsf{s}w,
\]
where $y$ is defined by \eqref{eq1344fri} and
\[
w(\theta)=\left(\frac{1+\sin\theta}{1-\sin\theta}\right)^{\mu/2}, \qquad \mu=\sgn k-\frac{2}{\pi}\arctan k.
\]
Here we have used \eqref{eq1528sat}.
Define
\begin{equation}\label{eq1330thu}
u_h := [\opP_r,\opM_{\mathsf s}](\mathsf{s}w)= \mathcal P_r\ast w - \mathsf{s}\mathcal P_r\ast(\mathsf{s}w).
\end{equation}
For a boundary datum $g$ and a scalar parameter $C$,  let $u_C(g)$ denote the corresponding solution.
With this notation,  $u_0(g)$ is the canonical solution represented by the kernel $\mathcal{K}^{(\sgn k)}$, and
\begin{equation}		\label{eq1253fri}
u_C(g)=u_0(g)-\frac{2\abs{k}}{\sqrt{k^2+1}}\,C u_h.
\end{equation}

The function $u_h$ is a weak solution of $Lu_h=0$ in $\mathbb D$.
Indeed, it is harmonic in each half-disk $\mathbb D_\pm$, and its traces across $\Gamma$ satisfy
\[
[u_h]_\Gamma=0, \qquad (\partial_x+k\mathsf{s}\partial_y)u_h=0 \quad\text{on }\;\Gamma.
\]
Moreover, $u_h$ has non-tangential limit zero at every point of $\partial\mathbb D$ except $\theta=\pm\pi/2$, and
\[
\mathcal N_*u_h\in L^p(\partial\mathbb D) \qquad\text{whenever}\qquad \abs{k}>\cot(\pi/2p).
\]

Suppose now that $C=C(g)$ is chosen so that $g\mapsto u_C(g)$ is linear and satisfies the natural $L^p$ non-tangential estimate.
Since both the canonical and the selected solution operators are bounded and $u_h \not \equiv0$, \eqref{eq1253fri} shows that $C$ is a bounded linear functional on $L^p(\partial\mathbb D;\mathbb R)$.
Hence, by $L^p$-duality,
\[
C(g)=\ip{g,\lambda}:= \frac1{2\pi}\int_{-\pi}^{\pi}g(\phi)\lambda(\phi)\,d\phi
\]
for some $\lambda\in L^{p'}(\partial\mathbb D;\mathbb R)$.
Consequently,
\[
u_C(r,\theta)= \frac{1}{2\pi} \int_{-\pi}^{\pi} \left\{ \mathcal K_k^{(\sgn k)}(r,\theta,\phi)- \frac{2\abs{k}}{\sqrt{k^2+1}}
u_h(r,\theta)\lambda(\phi) \right\}g(\phi)\,d\phi.
\]
The corresponding kernel is therefore
\[
\mathcal K^{(\sgn k)}_k(r,\theta, \phi) -\frac{2\abs{k}}{\sqrt{k^2+1}} u_h(r, \theta)\lambda(\phi),\qquad \lambda \in L^{p'}(\partial \mathbb D; \mathbb R).
\]
When $\abs{k}>\cot(\pi/2p)$, this modified kernel defines a solution operator satisfying the non-tangential maximal estimate.
Conversely, when $\abs{k}>\cot(\pi/2p)$, every $\lambda\in L^{p'}(\partial\mathbb D;\mathbb R)$ defines a bounded linear selection through this formula.
We now show that every such modified kernel must change sign.

We consider the case $k>0$; the case $k<0$ is similar.
We claim that no choice of
$\lambda\in L^{p'}(\partial\mathbb D;\mathbb R)$ can make the modified kernel nonnegative when
\[
k>\cot(\pi/2p) \iff \mu <1/p \iff 1-\mu>1/p'.
\]

An explicit computation gives
\[
u_h(r,\pi/2) = \frac{\sqrt{k^2+1}}{k} \left(\frac{1+r}{1-r}\right)^\mu.
\]
On the other hand, for $\phi\neq\pm\pi/2$, a direct calculation yields
\[
\mathcal K_k^{(1)}(r,\pi/2,\phi) =
-\frac{1}{\sqrt{k^2+1}} \left(\frac{1+r}{1-r}\right)^\mu
\left(\frac{1+\sin\phi}{1-\sin\phi} \right)^{(1-\mu)/2} \frac{2r(1-\sin\phi)}{1-2r\sin\phi+r^2}.
\]
Consequently,
\[
\lim_{r\nearrow 1} \frac{\mathcal K_k^{(1)}(r,\pi/2,\phi)}{u_h(r,\pi/2)} =-\frac{k}{k^2+1}
\left(\frac{1+\sin\phi}{1-\sin\phi}\right)^{(1-\mu)/2}.
\]
Suppose, to the contrary, that the modified kernel is nonnegative for every $0<r<1$ and $\theta\in\partial\mathbb D$, and for almost every $\phi$.
Since $u_h(r, \pi/2)>0$, evaluating at $\theta=\pi/2$, dividing by $u_h(r,\pi/2)$ and using the preceding limit, we obtain
\[
-\lambda(\phi) \ge b(\phi) :=
\frac{1}{2\sqrt{k^2+1}} \left(\frac{1+\sin\phi}{1-\sin\phi}\right)^{(1-\mu)/2}
\]
for almost every $\phi$.
The function $b$ is nonnegative and satisfies
\[
b(\phi) \asymp \abs{\phi-\pi/2}^{\mu-1} \qquad\text{as }\phi\to \pi/2.
\]
Since $1-\mu>1/p'$, it follows that $b\notin L^{p'}(\partial\mathbb D)$.
However, $-\lambda\ge b$ implies $\abs{\lambda}\ge b$, contradicting $\lambda\in L^{p'}(\partial\mathbb D)$.
Thus no modified kernel associated with $\lambda\in L^{p'}$ can be nonnegative when $k>\cot(\pi/2p)$.

Finally, choose $\phi\neq\pm\pi/2$ such that $\lambda(\phi)$ is finite. Then
\[
\mathcal K_k^{(1)}(r,\phi,\phi) = \mathcal P_r(0)+O(1) = \frac{1+r}{1-r}+O(1) \qquad\text{as }\;r\nearrow1,
\]
whereas $u_h(r,\phi)\to 0$.
Hence the modified kernel is positive at $(r,\phi,\phi)$ for $r$ sufficiently close to $1$.
Since it cannot be nonnegative everywhere, it must take both positive and negative values.

\subsection{Summary}
We conclude this section by comparing the admissible Riemann--Hilbert indices in the unit disk with those in Axelsson's half-plane model.
In the disk, local integrability at the two singular boundary points requires
\[
-1<\mu_N<1,\qquad \mu_N=-\frac{2}{\pi}\arctan k+N.
\]
In Axelsson's half-plane model, the corresponding exponent $\alpha$ is determined by
\[
-2<\alpha<1,\qquad \tan(\pi\alpha/2)=k.
\]
Thus, for $k \neq 0$, the two admissible branches differ by an index shift of two in the half-plane model, whereas they differ by one in the unit disk.
This difference shows that the unit disk model is not merely a bounded-domain analogue of the half-plane model, but has a genuinely different admissible branch structure.

\medskip
We summarize the preceding discussion in the following theorem.
\begin{theorem}	\label{thm:kernel}
Let $k\in\mathbb{R}$, and define
\[
\mu_N=-\frac{2}{\pi}\arctan k+N,
\qquad
p_k:=\frac{\pi}{\pi-2\arctan \,\abs{k}}.
\]
For each $N\in\mathbb Z$ such that $\mu_N\in(-1,1)$, there exists a real-valued Poisson kernel
\[
\mathcal K^{(N)}_k=\mathcal K_k^{(N)}(r, \theta,\phi)
\]
associated with the Dirichlet problem \eqref{eq:dp_in_disk}.
Let $g\in L^p(\partial\mathbb D)$, where
\[
p_k<p\le\infty \quad\text{if }\;N=0,
\qquad 1<p<p_k \quad\text{if }\;k\neq0\;\text{ and }\;N=\sgn k.
\]
Define
\[
u=\mathsf{P_r} g-k[\mathsf{P_r},\mathsf{M_s}] \tilde{g}, \qquad \tilde{g}=2i\opH y,
\]
where $y$ is given by \eqref{eq1910mon} when $N=0$, and by \eqref{eq1344fri} when $N=\pm1$.
Then $u$ solves \eqref{eq:dp_in_disk}, admits the representation
\[
u(r,\theta)= \frac{1}{2\pi}\int_{\partial\mathbb D} \mathcal K_k^{(N)}(r, \theta,\phi)g(\phi)\,d\phi,
\]
and satisfies
\[
\norm{\mathcal N_*u}_{L^p(\partial\mathbb D)} \le C_{k,p,\beta} \,\norm{g}_{L^p(\partial\mathbb D)},
\]
where $\beta$ is the aperture constant.
For each $1<p<\infty$, the following dichotomy holds.
\begin{enumerate}
\item
If $\abs{k}<\cot(\pi/2p)$, equivalently $p>p_k$, then the admissible $L^p$ branch is $N=0$, and the $L^p$ Dirichlet problem is uniquely solvable.
Moreover, $\mathcal K_k^{(0)}$ is nonnegative, and the corresponding solution satisfies the weak maximum principle for boundary data in $H^{1/2+}(\partial\mathbb D)$.

\item
If $\abs{k}>\cot(\pi/2p)$, equivalently $p<p_k$, then the admissible branch is $N=\sgn k$.
The homogeneous solution $u_h$ defined in \eqref{eq1330thu} shows that the $L^p$ Dirichlet problem is not uniquely solvable.
The kernel $\mathcal K_k^{(\sgn k)}$ defines the canonical solution operator.
Every solution differs from the canonical solution by a scalar multiple of $u_h$, determined by the free Riemann--Hilbert parameter $C$.
A solution selection depends linearly and boundedly on $g$ precisely when
\[
C=\ip{g,\lambda},
\qquad
\lambda\in L^{p'}(\partial\mathbb D;\mathbb R).
\]
Its kernel is then
\[
\mathcal K_k^{(\sgn k)}(r,\theta,\phi)- \frac{2\abs{k}}{\sqrt{k^2+1}}\, u_h(r,\theta)\lambda(\phi),
\]
and every such kernel changes sign.
\end{enumerate}
\end{theorem}

\noindent\textbf{Acknowledgments.}
The author was supported by the National Research Foundation of Korea (NRF) under agreement NRF-2022R1A2C1003322.


\end{document}